\documentclass[a4paper]{amsart}
\usepackage{amssymb,verbatim,dsfont}
\usepackage{xypic}
\xyoption{all}
\usepackage{verbatim}

\def \xcirc{\objectmargin{0.1pc}\def\objectstyle{\sssize}\diagram
\squarify<1pt>{}\circled\enddiagram}

\newtheorem{theorem}{Theorem}[section]
\newtheorem{lemma}[theorem]{Lemma}
\newtheorem{proposition}[theorem]{Proposition}
\newtheorem{corollary}[theorem]{Corollary}

\theoremstyle{definition}
\newtheorem{definition}[theorem]{Definition}

\newtheorem{example}[theorem]{Example}

\theoremstyle{remark}
\newtheorem{remark}[theorem]{Remark}

\DeclareMathOperator{\CMod}{\mathsf{CMod}}

\DeclareMathOperator{\dg}{\mathsf{dg}}

\DeclareMathOperator{\ev}{\mathsf{ev}}

\DeclareMathOperator{\ide}{\mathsf{id}}

\DeclareMathOperator{\rank}{\mathsf{rank}}

\DeclareMathOperator{\bN}{\mathbb{N}}

\DeclareMathOperator{\bZ}{\mathbb{Z}}

\newcommand{\ot}{\otimes}

\def\xcirc{\objectmargin{0.1pc}\def\objectstyle{\sssize}\diagram
\squarify<1pt>{}\circled\enddiagram}

\begin{document}

\title{Twisted planes}


\author{Jorge A. Guccione}
\address{Departamento de Matem\'atica\\ Facultad de Ciencias Exactas y Naturales, Pabell\'on 1
- Ciudad Universitaria\\ (1428) Buenos Aires, Argentina.} \curraddr{} \email{vander@dm.uba.ar}
\thanks{Supported by  UBACYT 0294 and  PIP 5617 (CONICET)}

\author{Juan J. Guccione}
\address{Departamento de Matem\'atica\\ Facultad de Ciencias Exactas y Naturales\\
Pabell\'on 1 - Ciudad Universitaria\\ (1428) Buenos Aires, Argentina.} \curraddr{}
\email{jjgucci@dm.uba.ar}
\thanks{Supported by  UBACYT 0294 and  PIP 5617 (CONICET)}

\author{Christian Valqui}
\address{Pontificia Universidad Cat\'olica del Per\'u - Instituto de Matem\'atica y Ciencias
Afi\-nes.} \curraddr{} \email{cvalqui@pucp.edu.pe}
\thanks{Supported by PUCP-DAI-3490.}

\begin{abstract}
Let $k$ be a commutative ring. We find and characterize a new family of twisted planes (i. e.
associative unitary $k$-algebra structures on the $k$-module $k[X,Y]$, having $k[X]$ and $k[Y]$
as subalgebras). Similar results are obtained for the $k$-module of two variables power series
$k[[X,Y]]$.
\end{abstract}

\date{}

\dedicatory{}


\maketitle

\section*{Introduction}
Let $k$ be a commutative ring and let $A$, $B$ be unitary $k$-algebras. By definition, a
twisted tensor product of $A$ with $B$ (over $k$) is an algebra structure defined on $A\ot_k
B$, with unit $1\ot_k 1$, such that the canonical maps $i_A\colon A\to A\ot_k B$ and $i_B\colon
B\to A\ot_k B$ are algebra maps satisfying $a\ot b= i_A(a)i_B(b)$. This structure has been
formerly studied by many people with different motivations (see for instance \cite{Ca},
\cite{C-S-V}, \cite{G-G}, \cite{Ma}, \cite{Tam}, \cite{VD-VK}). On one hand it is the most
general solution to the problem of factorization of structures in the setting of associative
algebras. Consequently, a number of examples of classical and recently defined constructions in
ring theory fits into this construction. For instance, Ore extensions, skew group algebras,
smash products, etcetera (for the definition and properties of these structures we refer to
\cite{Mo} and \cite{Ka}). On the other hand it has been proposed as the natural representative
for the cartesian product of nonconmutative spaces, this being based on the existing duality
between the categories of algebraic affine spaces and commutative algebras, under which the
cartesian product of spaces corresponds to the tensor product of algebras. And last, but not
least, twisted tensor products arise as a tool for building algebras starting with simpler
ones.

Given algebras $A$ and $B$, a basic problem is to determine all the twisted tensor products of
$A$ with $B$ and classify them up to a natural equivalence relation. A {\em (noncommutative)
polynomial extension} of a $k$-algebra $B$ is a twisted tensor product of a polynomial ring
$k[Y]$ with $B$. A {\em twisted plane} is such an extension in which $B$ is also a polynomial
algebra $k[X]$. That is, an associative unitary algebra $C$, with underlying $k$-module
$k[X,Y]$, such that:

\begin{itemize}

\smallskip

\item the natural inclusions $i_{k[X]}\colon k[X]\to C$ and $i_{k[Y]}\colon k[Y]\to
C$ are algebra maps,

\smallskip

\item $i_{k[X]}(X^m)i_{k[Y]}(Y^n)=X^mY^n$ for each $n,m\ge 0$.

\smallskip

\end{itemize}
For instance, Ore extensions of $k[Y]$ are examples of twisted planes. The aim of this paper is
to begin the study of these extensions, with emphasis in the problem of the classification of
the twisted planes. Actually, we do not solve completely this problem in the present work, but
we give the first step on having found and characterized a new family of twisted planes. Besides the twisted
polynomial extensions, in this article we also consider twisted extensions of the power series
ring $k[[X]]$, finding a new family of twisted tensor products of $k[[X]]$ with $k[[Y]]$.
Indeed, the natural setting to deal with adelically complete algebras such as $k[[X]]$ is the
monoidal category of filtered $k$-modules which are complete with respect to the induced
topology. Consequently, in this case we look for adelically complete algebras $C$, with
underlying topological $k$-module $k[[X,Y]]$, such that:

\begin{itemize}

\smallskip

\item the natural inclusions $i_{k[[X]]}\colon k[[X]]\to C$ and $i_{k[[Y]]}\colon
k[[Y]]\to C$ are continuos algebra maps,

\smallskip

\item $i_{k[[X]]}(X^m)i_{k[[Y]]}(Y^n)=X^mY^n$ for each $n,m\ge 0$.

\smallskip

\end{itemize}

\smallskip

From now on we assume implicitly that all the maps are $k$-linear maps, all the algebras are
over $k$, and the tensor product over $k$ is denoted $\ot$, without any subscript.

\smallskip

The paper is organized as follows: In Section~1 we have compiled without proofs some of the
standard facts on twisted tensor products, thus making our exposition self-contained. In
particular we recall the definition of a twisting map $s\colon A\ot B\to B\ot A$ and we
establish the bijective correspondence $s\mapsto B\ot_s A$ between twisting maps and twisted
tensor products. We also set up notation and terminology. In Section~2 we begin the study of
the noncommutative polynomial extensions. Consider an algebra $A$ and maps $\alpha_j\colon A\to
A$ ($j\ge 0$). In Theorem~\ref{teorema 2.1}, we determine necessary and sufficient conditions
for the existence of a (necessarily unique) twisting map $s\colon k[Y]\ot A\to A\ot k[Y]$ such
that
$$
s(Y\ot a)=\sum_{j=0}^{\infty} \alpha_j(a)\ot Y^j.
$$
When $\alpha_j=0$ for all $j\ge 2$, then we reobtain the familiar conditions to build an Ore
extension of $A$. That is, $\alpha_1$ must be an algebra endomorphism and  $\alpha_0$ must be
an $(\alpha_1,\ide)$-derivation. After that we give several examples, and later on, in
Theorem~\ref{teorema 2.7} and Corollary~\ref{corolario 2.8}, we establish a method to construct
a twisting map with $\alpha_0=0$ and $\alpha_1=\ide$ beginning with a locally nilpotent
derivation. Section~3 is devoted to the study of twisted planes. Theorem~3.1 and 3.4 are two of
the main results of this paper. Applying them, in Corollary~\ref{corolario 3.6} we obtain all
the twisting maps
$$
s\colon k[Y]\ot k[X]\to k[X]\ot k[Y]
$$
such that $\alpha_0=0$, $\alpha_1$ is the evaluation at an element of $k$ and $\{n: \alpha_n
\ne 0\}$ is  finite. The aim of Section~4 is to determine all the twisting maps
$$
s\colon k[Y]\ot k[t]/\langle t^2\rangle \to k[t]/\langle t^2\rangle\ot k[Y].
$$
To do this we first study the twisted tensor products $k[t]/\langle t^2\rangle \ot_s A$, then
we consider in detail the case $A=k[Y]$, and use that $s$ is a twisting map if and only if
$\tau\xcirc s\xcirc \tau$ is, where $\tau$ denotes the flip. Finally, in Section~5 we begin the
study of the twisted tensor products of the power series ring $k[[Y]]$ with an algebra $A$, in
the monoidal category of complete filtered $k$-modules. In this case, each map
$$
s\colon k[[Y]]\hat{\ot} A \to A\hat{\ot} k[[Y]]
$$
(where $\hat{\ot}$ denotes the completed tensor product over $k$) is also determined by a
family of maps $\alpha_j\colon A\to A$ ($j\ge 0$), but the conditions that these maps must
satisfy to guarantee that  $s$ is a twisting map, which are found in Theorem~\ref{teorema 5.3},
are somewhat different from those required when dealing with noncommutative polynomial
extensions. In Theorem~\ref{teorema 5.4} we give a version for complete algebras of
Theorem~\ref{teorema 2.7}, but the main result of this section, and one of the main results of
the paper, is Theorem~\ref{teorema 5.6}, in which we obtain all the twisting maps
$$
s\colon k[[Y]]\hat{\ot} A \to A\hat{\ot} k[[Y]]
$$
with $\alpha_0=0$.

\subsection*{Acknowledgment} This research was began during a visit of the first two authors to the
``IMCA'' and the ``PUCP''. They specially thank to these institutions for their hospitality and
support during their visit. We also thank Professor Guillermo Corti\~nas for useful comments.

\section{Preliminaries}
\setcounter{equation}{0}
In this section we review some of the basic facts about twisted tensor products. For their
proofs we refer to \cite{C-S-V}, \cite{VD-VK} and \cite{C-I-M-Z}. Given an algebra $A$ we let
$\eta_A$ and $\mu_A$ denote the unit and the multiplication maps of $A$, respectively.

\smallskip

Let $A$ and $B$ be algebras. A {\em twisted tensor product} of $A$ with $B$ is an algebra
structure on the $k$-module $A\ot B$, such that the canonical maps
$$
i_A\colon A\to A\ot B\quad\text{and}\quad i_B\colon B\to A\ot B
$$
are algebra homomorphisms and $\mu\xcirc (i_A\ot i_B) = \ide_{A\ot B}$, where $\mu$ denotes the
multiplication map of the twisted tensor product.

\smallskip

Assume we have a tensor product of $A$ with $B$. Then, the map
$$
s\colon B\ot A\to A\ot B,
$$
define by $s: = \mu \xcirc (i_B\ot i_A)$, satisfies:

\begin{enumerate}

\smallskip

\item $s\xcirc (\eta_B\ot A) = A\ot \eta_B$ and $s\xcirc (B\ot \eta_A) = \eta_A\ot B$,

\smallskip

\item $s\xcirc (\mu_B\ot A) = (A\ot \mu_B)\xcirc (s\ot B)\xcirc (B\ot s)$,

\smallskip

\item $s\xcirc (B\ot \mu_A) = (\mu_A\ot B)\xcirc (A\ot s)\xcirc (s\ot A)$.

\smallskip

\end{enumerate}
A map satisfying these conditions is call a {\em twisting map}. Conversely, if
$$
s\colon B\ot A\to A\ot B
$$
is a twisting map, then $A\ot B$ becomes a twisted tensor product via
$$
\mu_s := (\mu_A\ot\mu_B)\xcirc (A\ot s \ot B).
$$
This algebra will be denoted $A\ot_s B$. Furthermore, these constructions are inverse one of
each other.

\smallskip

The twisted tensor product $A\ot_s B$ has the following universal property: Given algebra maps
$f\colon A\to C$ and $g\colon B\to C$ such that
$$
\mu_C\xcirc (g\ot f) = \mu_C\xcirc (f\ot g)\xcirc s,
$$
there is a unique morphism of algebras $h\colon A\ot_sB\to C$ satisfying
$$
f = h\xcirc i_A\qquad\text{and}\qquad g = h\xcirc i_B.
$$
Indeed, it is easy to check that $h = \mu_C\xcirc (f\ot g)$.

\smallskip

The twisting maps are the objects of a category. Let $s\colon B\ot A\to A\ot B$ and $t\colon
D\ot C\to C\ot D$ be twisting maps. A morphism $(f,g)\colon s\to t$ is a pair  of morphism of
algebras $f\colon A\to C$ and $g\colon B\to D$ such that
$$
t\xcirc (g\ot f)=(f\ot g)\xcirc s.
$$
The composition is the evident one. Two twisting maps $s,t\colon B\ot A\to A\ot B$ are said to
be {\em equivalent} if they are isomorphic. That is, if there exist automorphisms $f\colon A\to
A$ and $g\colon B\to B$ such that $t= (f^{-1}\ot g^{-1})\xcirc s \xcirc (g\ot f)$.

\smallskip

The following result is useful to check that a map $s\colon B\ot A\to A\ot B$ is a twisting
map, and will be used implicitly in this paper.

\begin{proposition}\label{proposicion 1.1} Let $s\colon B\ot A\to A\ot B$ be a map satisfying
conditions~(1) and~(2). If $(b_i)_{i\in I}$ generates $B$ as an algebra and
$$
s(b_i\ot aa')=(\mu_A\ot B)\xcirc (A\ot s)\xcirc (s\ot A)(b_i\ot a\ot a')
$$
for all $a,a'\in A$ and each index $i$, then $s$ is a twisting map.
\end{proposition}

In the last section of this paper we will consider twisting maps between complete filtrated
algebras. Hence, we will work in the monoidal category $\CMod$, of complete filtrated
\mbox{$k$-modules}, where $k$ is a commutative ring. An object of $\CMod$ is a $k$-module $M$
endowed with a filtration
$$
M=M_0\supseteq M_1\supseteq M_2\supseteq\cdots,
$$
such that each $M_i$ is a $k$-module and $M$ is complete with respect to the topology induced
by the filtration. A morphism in $\CMod$ is a continuous  map $f\colon M\to N$ (namely, a  map
$f$ satisfying the requirement that for each $i\ge 0$ there exists $n_i$ such that
$f(M_{n_i})\subseteq N_i$). The tensor product of $M$ with $N$ in $\CMod$, denoted by
$M\hat{\ot} N$, is the completation of the usual tensor product $M\ot N$, with respect to the
topology induced by the filtration
$$
(M\ot N)_i= \sum_{r+s=i}\jmath (M_r\ot N_s),
$$
where $\jmath\colon M_r\ot N_s\to M\ot N$ is the canonical map.

Standard modules are considered as objects of $\CMod$ via the filtration
$$
M=M_0\supseteq 0= 0 =\cdots\quad \text{($M_i=0$ for all $i>0$).}
$$
The power series ring $k[[Y]]$ is an algebra in $\CMod$ via the usual filtration
$$
k[[Y]]= K[[Y]]\supseteq Y k[[Y]]\supseteq Y^2k[[Y]]\supseteq\cdots.
$$
Moreover, the completed tensor product  $k[[Y]]\hat{\ot} M$ is canonically isomorphic to
$M[[Y]]$, for each standard module $M$.

\smallskip

All the discussion preceding Proposition~\ref{proposicion 1.1} is valid for arbitrary monoidal
categories.

\section{Non-commutative polynomial extensions}
\setcounter{equation}{0}
This section is devoted to the study of the twisting maps $k[Y]\ot A\to A\ot k[Y]$, where $A$
is an arbitrary algebra. Given a family of maps $(\alpha_j\colon A\to A)_{j\ge 0}$ and indices
$n_1,\dots, n_r\ge 0$, we set $|n_1,\dots,n_r| = n_1+\cdots+n_r$ and $\alpha_{n_1\dots
n_r}=\alpha_{n_1}\xcirc \cdots \xcirc \alpha_{n_r}$. Moreover we write
$$
\gamma_j^{(0)} = \delta_{0j}\ide\quad\text{and}\quad \gamma_j^{(r)} = \sum_{|n_1,\dots, n_r|=j}
\alpha_{n_1\dots n_r}\,\text{ for $r>0$.}
$$
Note that $\gamma_j^{(1)} = \alpha_j$.

\setcounter{equation}{0}

\begin{theorem}\label{teorema 2.1} Let $A$ be an algebra and $s\colon k[Y]\ot A\to A\ot k[Y]$
a twisting map. The equation
$$
s(Y\ot a) = \sum_{j=0}^{\infty} \alpha_j(a)\ot Y^j,
$$
defines a family of maps $\alpha_j\colon A\to A$, which satisfies:

\begin{enumerate}

\smallskip

\item For each $a\in A$ there exists $j_0\ge 0$, such that $\alpha_j(a) = 0$ whenever $j>j_0$.

\smallskip

\item $\alpha_j(1) = \delta_{j1}$, where $\delta_{j1}$ denotes the symbol of Kronecker.

\smallskip

\item For all $j\ge 0$ and all $a,b\in A$,
\begin{equation}
\alpha_j(ab) = \sum_{r=0}^{\infty}\alpha_r(a)\gamma_j^{(r)}(b).\label{eq0}
\end{equation}

\smallskip

\end{enumerate}
Moreover,
\begin{equation}
s(Y^r\ot a)= \sum_{j=0}^{\infty}\gamma_j^{(r)}(a)\ot Y^j\label{eq1}
\end{equation}
for all $r\ge 0$ and $a\in A$. Conversely, given maps $\alpha_j\colon A\to A$ ($j\ge 0$)
satisfying (1)--(3), the formula~\eqref{eq1} defines a twisting map.

\end{theorem}

\begin{proof} Let $s$ be a twisting map. The formula for $s(Y^r\ot a)$ can be checked easily
by induction on $r$, using that $s(1\ot a)= a\ot 1$ and the compatibility of $s$ with the
multiplication of $k[Y]$. Item~(1) is immediate and items~(2) and (3) are consequences of the
compatibility of $s$ with the unit and the multiplication of $A$. Conversely, assume we have a
family of maps $(\alpha_j)_{j\ge 0}$ satisfying (1), (2) and (3). Let
$$
f\colon A[Y]\to A[Y]
$$
be the map given by $f(aY^j) = \sum_{i=0}^{\infty} \alpha_i(a)Y^{i+j}$, which is well defined
by (1). Since $f^r(a) = \sum_{j=0}^{\infty} \gamma_j^{(r)}(a)Y^j$ and obviously $f^r(a)\in
A[Y]$,  for each $a\in A$ and $r>0$, there exists $n\ge 0$ such that $\gamma_j^{(r)}(a)=0$
whenever $j>n$. This establishes the well-definition of formula~\eqref{eq1}. We leave the proof
that $s$ is a twisting map to the reader.
\end{proof}

\begin{remark}\label{remark 2.2} If $s\colon k[Y]\ot A\to A\ot k[Y]$ is a twisting map, then
$\ker(\alpha_0)$ is a subalgebra of $A$. Moreover,

\begin{itemize}

\smallskip

\item If $\alpha_0=0$, then $\alpha_1$ is a endomorphism of algebras.

\smallskip

\item Let $\nu>1$. If $\alpha_0 = 0$ and $\alpha_j = 0$ for $1<j<\nu$, then
$$
\alpha_{\nu}(ab)= \alpha_1(a)\alpha_{\nu}(b) + \alpha_{\nu}(a)\alpha_1^{\nu}(b).
$$

\end{itemize}

\end{remark}

\begin{example}\label{ejemplo 2.3} If $\alpha\colon A\to A$ is an algebra endomorphism and
$\delta\colon A\to A$ is an $\alpha$-derivation (that is $\delta(ab)=\delta(a)b
+\alpha(a)\delta(b)$), then there is a unique twisting map $s\colon k[Y]\ot A\to A\ot k[Y]$
such that
$$
s(Y\ot a)=\alpha(a) \ot Y + \delta(a)\ot 1\quad \text{for all $a\in A$.}
$$

\end{example}

\begin{example}\label{ejemplo 2.4} Let $A=k[t]/\langle t^2\rangle$. Consider the family of
maps $(\alpha_j\colon A\to A)_{j\ge 0}$, defined by
$$
\alpha_0=0,\quad \alpha_1=\ide,\quad \alpha_2(\lambda+\mu t)=\mu t\quad \text{and}\quad
\alpha_j =0\text{ for $j>2$.}
$$
The formula $s(Y\ot a)=\alpha_1(a)\ot Y+\alpha_2(a)\ot Y^2$ defines a twisting map.
\end{example}

Let $\alpha\colon A \to A$ be an algebra automorphism and let $(\beta_i\colon A\to A)_{i\ge 1}$
be a family of  maps. For $i_1,\dots,i_l\ge 1$, let
$$
\beta_{(i_1,\dots,i_l)} = \beta_{i_1}\xcirc \alpha^{-1}\xcirc \beta_{i_2}\xcirc \alpha^{-1}
\xcirc\cdots\xcirc\beta_{i_{l-2}}\xcirc\alpha^{-1}\xcirc\beta_{i_{l-1}}\xcirc\alpha^{-1}\xcirc
\beta_{i_l}
$$
Note that $\beta_{(i)} = \beta_i$. If $i_1,\dots,i_l = 1$ we will write $\beta_{(1)}^{(l)}$
instead of $\beta_{(1,\dots,1)}$. In particular $\beta_{(1)}^{(1)} = \beta_{(1)} = \beta_1$. We
also write $\beta_{(1)}^{(0)} = \alpha$.

\begin{lemma}\label{lema 2.5} Let $(\alpha_j\colon A\to A)_{j\ge 0}$ be the family of maps
defined by $\alpha_0 = 0$, $\alpha_1 = \alpha$ and
$$
\alpha_j=\sum_{l=1}^{j-1}\sum_{|i_1,\dots,i_l|=j-1}\beta_{(i_1,\dots,i_l)}\qquad\text{for $j\ge
2$.}
$$
Then, for all $j\ge r$,
\begin{equation*}
\gamma_j^{(r)} = L + \sum_{n_1,\dots n_r\ge 0\atop |n_1,\dots,n_r| = j-r} \beta_{(1)}^{(n_1)}
\xcirc\cdots\xcirc\beta_{(1)}^{(n_r)},
\end{equation*}
where $L$ is sum of compositions of $\alpha$'s, $\alpha^{-1}$'s and $\beta_i$'s, in which at
least one $\beta_i$ with $i>1$, appears.
\end{lemma}

\begin{proof} Since
$$
\alpha_j = \beta_1^{(j-1)} + \sum_{l=1}^{j-2}\sum_{|i_1,\dots,i_l|=j-1}\beta_{(i_1,\dots,i_l)}
\qquad\text{for all $j\ge 1$,}
$$
we have
\begin{align*}
\gamma_j^{(r)} & = \sum_{|n_1,\dots,n_r| = j} \alpha_{n_1\dots n_r}\\
& = L+\sum_{|n_1,\dots,n_r|=j}\beta_{(1)}^{(n_1-1)}\xcirc\cdots\xcirc\beta_{(1)}^{(n_r-1)},\\
& = L+\sum_{n_1,\dots n_r\ge 0\atop |n_1,\dots,n_r|=j-r}\beta_{(1)}^{(n_1)}\xcirc\cdots \xcirc
\beta_{(1)}^{(n_r)},
\end{align*}
as desired.
\end{proof}

Let $A$ be an algebra and $\varphi$, $\psi$ endomorphisms of $A$. Recall that a map $d\colon
A\to A$ is a $(\varphi,\psi)$-derivation if
$$
d(ab)=d(a)\psi(b)+ \varphi(a)d(b)\text{ for all $a,b\in A$.}
$$

\begin{lemma}\label{lema 2.6} For each $i\ge 1$, let $\beta_i\colon A\to A$ be an
$(\alpha,\alpha^{i+1})$-derivation. Assume that if $i+i'\ge 3$, then
$\alpha^r(\beta_i(a))\beta_{i'}(b)=0$ for all $r\in\bZ$ and $a,b\in A$. We have:

\begin{enumerate}

\smallskip

\item $\beta_{(i_1,\dots,i_l)}(a)L(b) = 0$, where $L$ is as in Lemma~\ref{lema 2.5}.

\smallskip

\item If some $i_u>1$ and some $n_v>0$, then
$$
\beta_{(i_1,\dots,i_l)}(a)\alpha^h\xcirc \beta_1(b) = \alpha^h\xcirc \beta_1(a)
\beta_{(i_1,\dots,i_l)}(b) = 0.
$$

\smallskip

\item If some $i_u>1$, then
$$
\beta_{(i_1,\dots,i_l)}\,\text{ is an $(\alpha,\alpha^j)$-derivation,}
$$
where $j = 1+i_1+\cdots+i_l$.

\end{enumerate}

\end{lemma}

\begin{proof} (1) let $f_1\xcirc\cdots \xcirc f_v$, where $f_i = \{\alpha,\alpha^{-1},
\beta_1,\beta_2,\dots \}$ be a term of $L$. Let $i_0$ be the least $i$ such that $f_i =
\beta_j$ with $j>1$. By definition $\beta_{(i_1,\dots,i_l)}(a) = \beta_{i_1}(a')$, where $a' =
\alpha^{-1} (\beta_{(i_2,\dots,i_l)}(a))$. If $i_0 = 1$, then
$$
\beta_{i_1}(a')\bigl(f_1\xcirc \cdots \xcirc f_v(b)\bigr) = 0,
$$
by hypothesis. The general case follows by induction on $i_0$ using that if $f_1 = \alpha^{\pm
1}$, then
$$
\beta_{i_1}(a')f_1\xcirc\cdots\xcirc f_v(b) = f_1\Bigl(f_1^{-1}\xcirc\beta_{i_1}(a') f_2\xcirc
\cdots \xcirc f_v(b)\Bigr)
$$
and if $f_1 = \beta_1$, then
\begin{align*}
\beta_{i_1}(a') f_1\xcirc\cdots\xcirc f_v(b) & = \beta_1\Bigl(\alpha^{-1}\xcirc\beta_{i_1}(a')
f_2 \xcirc \cdots \xcirc f_v(b)\Bigr)\\
& + \alpha^2\Bigl(\alpha^{-2}\xcirc\beta_1\xcirc\alpha^{-1}\xcirc\beta_{i_1}(a')f_2\xcirc\cdots
\xcirc f_v(b)\Bigr).
\end{align*}

\smallskip

\noindent (2) It is similar to (1).

\smallskip

\noindent (3) We make the proof by induction on $l$. First assume that $u>1$. Then, by the
inductive hypothesis,
\begin{align*}
\beta_{(i_1,\dots,i_l)}(ab) & = \beta_{i_1}\xcirc\alpha^{-1}\xcirc\beta_{(i_2,\dots,i_l)}(ab)\\
& = \beta_{i_1}\xcirc  \alpha^{-1}\Bigl(\beta_{(i_2,\dots,i_l)}(a)\alpha^{1+i_2+\cdots+i_l}(b)
+\alpha(a)\beta_{(i_2,\dots,i_l)}(b) \Bigr)\\
& = \beta_{(i_1,\dots,i_l)}(a)\alpha^{1+i_1+\cdots+i_l}(b) + \beta_{(i_2,\dots,i_l)}(a)
\beta_{i_1} \xcirc \alpha^{i_2+\cdots+i_l}(b)\\
& +\beta_{i_1}(a)\alpha^{i_1}\xcirc \beta_{(i_2,\dots,i_l)}(b) + \alpha(a)
\beta_{(i_1,\dots,i_l)}(b)\\
& =\beta_{(i_1,\dots,i_l)}(a)\alpha^{1+i_1+\cdots+i_l}(b)+\alpha(a) \beta_{(i_1,\dots,i_l)}(b),
\end{align*}
where the last equality follows from the fact that, by item~(2),
$$
\beta_{(i_2,\dots,i_l)}(a)\beta_{i_1}\xcirc\alpha^{i_2+\cdots+i_l}(b) = \beta_{i_1}(a)
\alpha^{i_1}\xcirc \beta_{(i_2,\dots,i_l)}(b) = 0.
$$
Assume now that $u=1$. Then, arguing as above we obtain,
\begin{align*}
\beta_{(i_1,\dots,i_l)}(ab) &=\beta_{(i_1,\dots,i_{l-1})}\xcirc\alpha^{-1}\xcirc
\beta_{i_l}(ab)\\
& = \beta_{(i_1,\dots,i_l)}(a)\alpha^{1+i_1+\cdots+i_l}(b)+\alpha(a)\beta_{(i_1,\dots,i_l)}(b),
\end{align*}
as desired.
\end{proof}

\begin{theorem}\label{teorema 2.7} Let $\alpha\colon A \to A$ be an algebra automorphism. For
each $i\ge 1$, let $\beta_i\colon A\to A$ be an $(\alpha,\alpha^{i+1})$-derivation. If

\begin{enumerate}

\smallskip

\item $\alpha^r(\beta_i(a))\beta_{i'}(b)=0$ for all $r\in\bZ$ and $a,b\in A$ whenever $i+i'\ge
3$,

\smallskip

\item For all $a\in A$ there is $n\in \bN$ such that
$$
\sum_{l=1}^j\sum_{|i_1,\dots,i_l|=j}\beta_{(i_1,\dots,i_l)}(a) = 0\quad\text{for all $j>n$,}
$$

\smallskip

\end{enumerate}
then, the formula
$$
s(Y\ot a)=\sum_{j=0}^{\infty} \alpha_j(a)\ot Y^j,
$$
where the maps $\alpha_j \colon A\to A$ ($j\ge 0$) are constructed as in Lemma~\ref{lema 2.5},
defines a twisting map $s\colon k[Y]\ot A\to A\ot k[Y]$.
\end{theorem}

\begin{proof} By item~(2), the maps $\alpha_j$ satisfy condition~(1) of Theorem~\ref{teorema
2.1}. Condition~(2) follows from the fact that the $\beta_i$'s are derivations. It remains to
check that condition~(3) also holds. For $j\le 1$ this is immediate. Assume $j\ge 2$ and set
$$
T  = \sum_{r=0}^{\infty}\alpha_r(a)\gamma_j^{(r)}(b).
$$
On one hand, by Lemma~\ref{lema 2.5} and~\ref{lema 2.6}, \allowdisplaybreaks
\begin{align*}
T & = \sum_{r=1}^j \alpha_r(a)\gamma_j^{(r)}(b)\\
& = \alpha_1(a)\gamma_j^{(1)}(b) + \alpha_j(a)\gamma_j^{(j)}(b) + \sum_{r=2}^{j-1}
\alpha_r(a)\gamma_j^{(r)}(b)\\
& =\sum_{l=1}^{j-1} \sum_{|i_1,\dots, i_l|=j-1} \alpha(a)\beta_{(i_1,\dots,i_l)}(b) +
\sum_{l=1}^{j-1} \sum_{|i_1,\dots, i_l|=j-1} \beta_{(i_1,\dots,i_l)}(a)\alpha^j(b)\\
&+ \sum_{r=2}^{j-1}\left(\sum_{l=1}^{r-1}\sum_{|i_1,\dots,i_l|=r-1}\beta_{(i_1,\dots,i_l)}(a)
\right)\left(L(b) +\sum_{|n_1,\dots, n_r|=j-r}\beta_{(1)}^{(n_1)}\xcirc\cdots\xcirc
\beta_{(1)}^{(n_r)}(b)\right)\\
& =\sum_{l=1}^{j-1} \sum_{|i_1,\dots, i_l|=j-1} \alpha(a)\beta_{(i_1,\dots,i_l)}(b) +
\sum_{l=1}^{j-1} \sum_{|i_1,\dots,i_l|=j-1} \beta_{(i_1,\dots,i_l)}(a)\alpha^j(b)\\
&+ \sum_{r=2}^{j-1} \sum_{|n_1,\dots,n_r|=j-r} \beta_{(1)}^{(r-1)}(a) \beta_{(1)}^{(n_1)}\xcirc
\cdots\xcirc\beta_{(1)}^{(n_r)}(b).
\end{align*}
On the other hand, since by item~(3) of Lemma~\ref{lema 2.6}, $\beta_{(i_1,\dots,i_l)}$ is an
$(\alpha,\alpha^j)$-derivation whenever some $i_u>1$, we have
\begin{align*}
\alpha_j(ab) & = \sum_{l=1}^{j-1}\sum_{|i_1,\dots,i_l|=j-1}\beta_{(i_1,\dots,i_l)}(ab)\\
& = \sum_{l=1}^{j-2}\sum_{|i_1,\dots,i_l|=j-1} \Bigl(\alpha(a) \beta_{(i_1,\dots,i_l)}(b)
+\beta_{(i_1,\dots,i_l)}(a)\alpha^j(b)\Bigr)+ \beta_{(1)}^{(j-1)}(ab).
\end{align*}
So, in order to finish the proof it suffices to show that for all $j\ge 2$,
\begin{align*}
\beta_{(1)}^{(j-1)}(ab) & =\alpha(a)\beta_{(1)}^{(j-1)}(b)+\beta_{(1)}^{(j-1)}(a) \alpha^j(b)\\
& + \sum_{r=2}^{j-1}\sum_{|n_1,\dots,n_r| =j-r} \beta_{(1)}^{(r-1)}(a)\beta_{(1)}^{(n_1)}
\xcirc \cdots\xcirc\beta_{(1)}^{(n_r)}(b).
\end{align*}
We proceed by induction on $j$. When $j = 2$,
$$
\beta_{(1)}^{(1)}(ab) = \beta_1(ab) = \alpha(a)\beta_1(b) + \beta_1(a)\alpha^2(b),
$$
since $\beta_1$ is an $(\alpha,\alpha^2)$-derivation. Assume that the result is valid for $j$.
Then,
\begin{align*}
\beta_{(1)}^{(j)}(ab) & = \beta_1\xcirc\alpha^{-1}\xcirc\beta_{(1)}^{(j-1)}(ab)\\
& = \beta_1\left(a\alpha^{-1}\xcirc\beta_{(1)}^{(j-1)}(b)\right) + \beta_1\left(\alpha^{-1}
\xcirc \beta_{(1)}^{(j-1)}(a)\alpha^{j-1}(b)\right)\\
&+\beta_1\left(\sum_{r=2}^{j-1}\sum_{|n_1,\dots,n_r| =j-r}\alpha^{-1}\xcirc
\beta_{(1)}^{(r-1)}(a)\alpha^{-1}\xcirc\beta_{(1)}^{(n_1)}\xcirc\cdots\xcirc
\beta_{(1)}^{(n_r)}(b) \right)\\
& = \alpha(a)\beta_{(1)}^{(j)}(b) + \beta_{(1)}^{(j)}(a) \alpha^{j+1}(b)\\
& + \beta_1(a)\alpha\xcirc \beta_{(1)}^{(j-1)}(b)+\beta_{(1)}^{(j-1)}(a) \beta_1\xcirc
\alpha^{j-1}(b)\\
& + \sum_{r=2}^{j-1}\sum_{|n_1,\dots,n_r| =j-r}\beta_{(1)}^{(r)}(a)\alpha\xcirc
\beta_{(1)}^{(n_1)} \xcirc\cdots\xcirc\beta_{(1)}^{(n_r)}(b)\\
& + \sum_{r=2}^{j-1}\sum_{|n_1,\dots,n_r| =j-r}\beta_{(1)}^{(r-1)}(a)\beta_{(1)}^{(n_1+1)}
\xcirc \cdots\xcirc\beta_{(1)}^{(n_r)}(b)\\
& = \alpha(a)\beta_{(1)}^{(j)}(b) + \beta_{(1)}^{(j)}(a)\alpha^{j+1}(b)\\
& + \sum_{r=2}^j \sum_{|n_1,\dots,n_r| =j+1-r} \beta_{(1)}^{(r-1)}(a) \beta_{(1)}^{(n_1)}\xcirc
\cdots\xcirc\beta_{(1)}^{(n_r)}(b).
\end{align*}
This finish the proof.
\end{proof}

\begin{corollary}\label{corolario 2.8} Let $A$ be an algebra, $\alpha\colon A \to A$ be an
algebra automorphism and $\beta\colon A\to A$ be an $(\alpha,\alpha^2)$-derivation. Let
$(\alpha_j\colon A\to A)_{j\ge 0}$ be the family of maps defined by $\alpha_0=0$ and $\alpha_j
= \bigl(\beta\xcirc\alpha^{-1}\bigr)^{j-2}\xcirc \beta$ for $j\ge 1$. If $\beta\xcirc
\alpha^{-1}$ is locally nilpotent (that is, for each $a\in A$ there exists $n\ge 1$ such that
$\bigl(\beta\xcirc\alpha^{-1}\bigr)^n(a)=0$), then the formula
$$
s(Y\ot a) = \sum_{j=0}^{\infty} \alpha_j(a)\ot Y^j
$$
defines a twisting map $s\colon k[Y]\ot A\to A\ot k[Y]$.
\end{corollary}

\begin{proof} Take $\beta_1 = \beta$ and $\beta_j = 0$ for $j>1$ in Theorem~\ref{teorema 2.7}.
\end{proof}

\begin{example}\label{ejemplo 2.9} Let $A=k[t]/\langle t^n\rangle$. Let $D\colon A\to A$ be
the derivation defined by $D(P)(t)=P'(t)t^2$. The formula
$$
s(Y\ot P)= P\ot Y + \sum_{j=1}^{\infty}  D^j(P)\ot Y^{j+1}
$$
defines a twisting map $s\colon k[Y]\ot A\to A\ot k[Y]$.
\end{example}

\section{Twisted planes}
\setcounter{equation}{0}
The aim of this section is to study in detail the twisting maps
$$
s\colon k[Y]\ot k[X]\to k[X]\ot k[Y].
$$

\begin{theorem}\label{teorema 3.1} Let $\sum_{ij} q_{ij} X^i\ot Y^j\in k[X]\ot k[Y]$. If
$q_{ij} = 0$ whenever $i\le 1$ or $j\le 1$, then there is a unique twisting map $s\colon
k[Y]\ot k[X]\to k[X]\ot k[Y]$ such that $s(Y\ot X) = \sum_{ij} q_{ij} X^i\ot Y^j$. Moreover
$s(Y^r\ot X^s) = 0$ whenever $r,s>0$ and $r+s>2$.
\end{theorem}

\begin{proof} First we assume that a such twisting map exists and we prove that it satisfies
$s(Y^r\ot X^s) = 0$ if $r,s>0$ and $r+s>2$. Let $(\alpha_j)_{j\ge 0}$ be as in
Theorem~\ref{teorema 2.1}. So, $\alpha_j(X) = \sum_i q_{ij}X^i$. The hypothesis means that
$\alpha_0(X)=\alpha_1(X)=0$ and $\alpha_j(X)\in X^2k[X]$ for each $j\ge 2$. Consequently, by
Remark~\ref{remark 2.2}, $\alpha_0 =0$ and $\alpha_1$ is the evaluation at $0$. We assert that
$\alpha_j(X^n)=0$ for each $j\ge 0$ and $n\ge 2$. For $j=0,1$ this is clear. Assume
$\alpha_l(X^n)=0$ for all $l<j$ and $n\ge 2$. Then, by item~(3) of Theorem~\ref{teorema 2.1},
we have
$$
\alpha_j(X^{2}) = \sum_{r=2}^j\alpha_r(X)\gamma_j^{(r)}(X),
$$
which vanishes, because clearly $\alpha_{n_1\dots n_r}(X)=0$ if $n_r\le 1$, and also if $n_r\ge
2$ since, in this case, $\alpha_{n_r}(X)\in X^2k[X]$ and $n_{r-1}<j$. Assuming now that $n\ge
3$ and $\alpha_j(X^n) =0$, using again item~(3) of Theorem~\ref{teorema 2.1}, we obtain
$$
\alpha_j(X^n) = \sum_{r=2}^j\alpha_r(X^2)\gamma_j^{(r)}(X^{n-2}) = 0.
$$
It is now easy to check that $s(Y^r\ot X^s) = 0$ whenever $r,s>0$ and $r+s>2$, as wanted.
Finally, to check the existence of $s$, it suffices to note that the family of maps
$(\alpha_j\colon A\to A)_{j\ge 0}$, defined by $\alpha_0 = 0$, $\alpha_1(X^n) = \delta_{1n}$
and
$$
\alpha_j(X^n) = \begin{cases} \sum_{i\ge 0} q_{ij}X^i &\text{if $n=1$,}\\ 0
&\text{otherwise,}\end{cases}
$$
for $j\ge 2$ satisfies the conditions required in Theorem~\ref{teorema 2.1}. We leave the
details to the reader.
\end{proof}

\begin{definition}\label{definicion 3.2} A twisting map $s\colon k[Y]\ot k[X]\to k[X]\ot
k[Y]$ is {\em upper bounded} if there exists $n_0\in \mathbb{N}$ such that $\alpha_n=0$ for all
$n\ge n_0$. It is {\em lower bounded} if $\tau\xcirc s\xcirc\tau$ is upper bounded, where
$\tau$ denotes the flip. Finally, we say that $s$ is {\em bounded} if it is upper and lower
bounded.
\end{definition}

\begin{example}\label{example 3.3} Twisting maps associated with Ore extensions
$k[X][Y,\alpha,\delta]$ are upper bound\-ed, but in general they are not lower bounded. The
twisting maps introduced in Theorem~\ref{teorema 3.1} are bounded.
\end{example}

By the sake of continuity the proof of the following result is relegated to an appendix.

\begin{theorem}\label{teorema 3.4} Assume that $k$ is a commutative domain. Let
$$
\sum_{ij} q_{ij} X^i\ot Y^j\in k[X]\ot k[Y].
$$
The following facts hold:

\begin{enumerate}

\smallskip

\item If $q_{i0}=q_{i1}=0$ for all $i\ge 0$ and there is a (necessarily unique) upper bounded
twisting map $s\colon k[Y]\ot k[X] \to k[X]\ot k[Y]$ such that
$$
s(Y \ot X)= \sum_{ij} q_{ij} X^i\ot Y^j,
$$
then $q_{0j}=q_{1j}= 0$ for all $j\ge 0$.

\smallskip

\item If $q_{0j}=q_{1j}= 0$ for all $j\ge 0$ and there is a (necessarily unique) lower bounded
twisting map $s\colon k[Y]\ot k[X] \to k[X]\ot k[Y]$ such that
$$
s(Y \ot X)= \sum_{ij} q_{ij} X^i\ot Y^j,
$$
then $q_{i0}=q_{i1}=0$ for all $i\ge 0$.

\end{enumerate}

\end{theorem}

We say that a twisting map $s\colon k[Y]\ot k[X]\to k[X]\ot k[Y]$ is {\em almost null} if it is
equivalent (in the sense introduced above Proposition~\ref{proposicion 1.1}) to one of the
twisting maps considered in Theorem~\ref{teorema 3.1}.

\begin{corollary}\label{corolario 3.5} Assume that $k$ is a commutative domain. Let
$$
\sum_{ij} q'_{ij} X^i\ot Y^j\in k[X]\ot k[Y].
$$
Consider the polynomials
$$
P_i(Z)= \sum_{n\ge 0} q'_{in}Z^n \quad\text{and}\quad Q_j(Z)= \sum_{m\ge 0} q'_{mj}Z^m\qquad
\text{($i,j\ge 0$).}
$$
The following facts hold:

\begin{enumerate}

\smallskip

\item There is an almost null twisting map $s'\colon k[X]\ot k[Y]\to k[Y]\ot k[X]$ such that
$$
s'(Y\ot X)= \sum_{ij} q'_{ij} X^i\ot Y^j,
$$
if and only if there exist $\lambda,\xi\in k$ satisfying:
\begin{enumerate}

\smallskip

\item[(a)] $P_0(\xi)=0$, $P_1(\xi)=\xi$, $P'_0(\xi)=\lambda$, $P'_1(\xi)=0$ and $\xi$ is
multiple root of $P_i(Z)$ for each $i>1$.

\smallskip

\end{enumerate}
and
\begin{enumerate}

\smallskip

\item[(b)] $Q_0(\lambda)=0$, $Q_1(\lambda)=\lambda$, $Q'_0(\lambda)=\xi$, $Q'_1(\lambda)=0$
and $\lambda$ is multiple root of $Q_j(Z)$ for each $j>1$.

\smallskip

\end{enumerate}
Moreover, the equivalence $s'\simeq s$ is realized by means of the automorphisms
$$
f\colon k[Y]\to k[Y]\quad\text{and}\quad g\colon k[X]\to k[X],
$$
defined by $f(Y) = Y-\xi$ and $g(X) = X-\lambda$.

\smallskip

\item If there exist $\lambda,\xi\in k$ that satisfy item~(a), but not item~(b), then there is
not an upper bounded twisting map $s'\colon k[Y]\ot k[X]\to k[X]\ot k[Y]$ such that
$$
s'(Y\ot X)= \sum_{ij} q'_{ij} X^i\ot Y^j.
$$

\smallskip

\item If there exist $\lambda,\xi\in k$ that satisfy item~(b), but not item~(a), then there is
not a lower bounded twisting map $s'\colon k[Y]\ot k[X]\to k[X]\ot k[Y]$ such that
$$
s'(Y\ot X)= \sum_{ij} q'_{ij} X^i\ot Y^j.
$$

\end{enumerate}

\end{corollary}

\begin{proof} It is easy to check that $s'\colon k[Y]\ot k[X] \to k[X]\ot k[Y]$ is an almost
null twisting map if and only if there exist $\lambda,\xi\in k$ such that
$$
s = (g^{-1}\ot f^{-1})\xcirc s'\xcirc (f\ot g)
$$
satisfies the conditions required in Theorem~\ref{teorema 3.1}, where
$$
f\colon k[Y]\to k[Y]\quad\text{and}\quad g\colon k[X]\to k[X]
$$
are the automorphisms defined by $f(Y)=Y-\xi$ and $g(X)= X-\lambda$. Write
$$
s(Y\ot X)=\sum_{ij} q_{ij} X^i\ot Y^j\qquad\text{and}\qquad s'(Y\ot X)=\sum_{ij} q'_{ij} X^i\ot
Y^j.
$$
A direct computation shows that
\begin{align*}
s(Y\ot X) & = \bigl((g^{-1}\ot f^{-1})\xcirc s'\xcirc (f\ot g)\bigr)(x\ot X)\\
&  = \sum_{ij}\!\! \left(\sum_{mn=0}^{\infty} \binom{m}{i} \binom{n}{j} \lambda^{m-i}\xi^{n-j}
q'_{mn}\!\!\right)\! X^i\ot Y^j-X\ot \xi -\lambda\ot Y - \lambda\ot \xi,
\end{align*}
where we adopt the usual convention that a combinatorial numbers are zero if its numerator is
lesser than its denominator. Clearly the following facts hold:

\smallskip

\noindent (1)\enspace $q_{i0} = q_{i1} = 0$ for all $i\ge 0$ if and only if $\lambda$ and $\xi$
satisfy:
\begin{alignat}{2}
&\sum_{mn=0}^{\infty} \lambda^m\xi^nq'_{mn} -\lambda\xi=0,\tag{3.1 a}\\
&\sum_{mn=0}^{\infty} m\lambda^{m-1}\xi^nq'_{mn} -\xi=0,\tag{3.1 b}\\
&\sum_{mn=0}^{\infty} \binom{m}{i}\lambda^{m-i}\xi^nq'_{mn}= 0 && \quad\text{for each
$i>1$},\tag{3.1 c}\\
&\sum_{mn=0}^{\infty} n\lambda^m\xi^{n-1} q'_{mn} -\lambda=0,\tag{3.1 d}\\
&\sum_{mn=0}^{\infty} \binom{m}{i} n \lambda^{m-i}\xi^{n-1} q'_{mn}&&\quad\text{for each
$i>0$}\tag{3.1 e}.
\end{alignat}

\smallskip

\noindent (2)\enspace $q_{0j} = q_{1j} = 0$ for all $j\ge 0$ if and only if $\lambda$ and $\xi$
satisfy:
\begin{alignat}{2}
&\sum_{mn=0}^{\infty} \lambda^m\xi^n q'_{mn} -\lambda\xi=0,\tag{3.2 a}\\
&\sum_{mn=0}^{\infty} n\lambda^m\xi^{n-1} q'_{mn} -\lambda=0,\tag{3.2 b}\\
&\sum_{mn=0}^{\infty} \binom{n}{j}  \lambda^m\xi^{n-j} q'_{mn}&&\quad\text{for each
$j>1$}\tag{3.2 c},\\
&\sum_{mn=0}^{\infty}  m\lambda^{m-1}\xi^nq'_{mn} -\xi=0,\tag{3.2 d}\\
&\sum_{mn=0}^{\infty} m\binom{n}{j} \lambda^{m-1}\xi^{n-j} q'_{mn}&&\quad\text{for each
$j>0$.}\tag{3.2 e}
\end{alignat}

\smallskip

\noindent It is easy to check that conditions (3.1 a)--(3.1 e) are equivalents to
\begin{alignat}{2}
& P_i(\xi) =  \sum_{n=0}^{\infty} \xi^n q'_{in} =0 &&\quad\text{for $i\ne 1$},\tag{3.3 a}\\
&P_1(\xi)-\xi =  \sum_{n=0}^{\infty} \xi^n q'_{1n} -\xi=0,\tag{3.3 b}\\
&P'_0(\xi) - \lambda = \sum_{n=0}^{\infty}  n \xi^{n-1} q'_{0n} -\lambda=0,\tag{3.3 c}\\
&P'_i(\xi) = \sum_{n=0}^{\infty} n \xi^{n-1} q'_{in}&&\quad\text{for $i>0$}\tag{3.3 d},
\intertext{and that conditions (3.2 a)--(3.2 e) are equivalents to}
& Q_j(\lambda) = \sum_{m=0}^{\infty} \lambda^m q'_{mj} =0 &&\quad\text{for $j\ne 1$},\tag{3.4
a}\\
& Q_1(\lambda) - \lambda = \sum_{m=0}^{\infty} \lambda^m q'_{m1} -\lambda=0,\tag{3.4 b}\\
& Q'_0(\lambda) - \xi=  \sum_{m=0}^{\infty} m \lambda^{m-1} q'_{m0} -\xi=0,\tag{3.4 c}\\
& Q'_j(\lambda) = \sum_{m=0}^{\infty} m \lambda^{m-1}q'_{mj}&&\quad\text{for $j>0$}\tag{3.4 d}.
\end{alignat}
So, equalities~(3.3 a)--(3.3 d) correspond to the hypothesis of Theorem~\ref{teorema 3.4}~(1)
and to the thesis of Theorem~\ref{teorema 3.4}~(2) and equalities~(3.4 a)--(3.4 d) correspond
to the hypothesis of Theorem~\ref{teorema 3.4}~(2) and to the thesis of Theorem~\ref{teorema
3.4}~(1). The proof can be easily finished using Theorem~\ref{teorema 3.4} and these remarks.
\end{proof}

\begin{corollary}\label{corolario 3.6} Let $\sum_{ij} q'_{ij} X^i\ot Y^j\in k[X]\ot k[Y]$. The
following facts hold:

\begin{enumerate}

\smallskip

\item If $q'_{i0} = 0$ for all $i$ and $q'_{i1} = 0$ for all $i\ge 1$, then there exists an
upper bounded twisting map $s'\colon k[Y]\ot k[X] \to k[X]\ot k[Y]$ such that
$$
s'(Y\ot X) = \sum_{ij} q'_{ij} X^i\ot Y^j,
$$
if and only if $q'_{01}$ is a multiple root of $Q_j(Z) = \sum\limits_{m\ge 0} q'_{mj} Z^m$ for
each $j>1$.

\smallskip

\item If $q'_{0j} = 0$ for all $j$ and $q'_{1j} = 0$ for all $j\ge 1$, then there exists an
lower bounded twisting map $s'\colon k[Y]\ot k[X] \to k[X]\ot k[Y]$ such that
$$
s'(Y\ot X) = \sum_{ij} q'_{ij} X^i\ot Y^j,
$$
if and only if $q'_{10}$ is a multiple root of $P_i(Z) = \sum\limits_{n\ge 0} q'_{in} Z^n$ for
each $i>1$.

\end{enumerate}

\end{corollary}

\begin{proof} Item~(2) follows immediately from item~(1), since $s'$ is a twisting map if and
only if $\tau \xcirc s'\xcirc \tau$ is, where $\tau$ is the flip. So, we are reduced to prove
the first item. This follows from Corollary~\ref{corolario 3.5}, since $(q'_{01},0)$ satisfies
the conditions asked for $(\lambda,\xi)$ in item~(1)(a) of that corollary and $(q'_{01},0)$
satisfies those required to $(\lambda,\xi)$ in item~(1)(b) if and only if $q'_{01}$ is a
multiple root of $Q_j(Z) = \sum_{m=0}^{\infty} q'_{mj} Z^m$ for each $j>1$.
\end{proof}

The Corollary gains in interest if we realize that in item~(1) we get all the upper bounded
twisting maps with $\alpha_0 = 0$ and $\alpha_1$ the evaluation at an element of $k$.

\section{Non-commutative extensions of the dual numbers}
\setcounter{equation}{0}
It seems very difficult to compute all the twisting maps $s\colon k[Y]\ot A\to A\ot k[Y]$ for a
particular algebra $A$. In this section we accomplish this for $A=k[t]/\langle t^2 \rangle$
using the evident fact that $s$ is a twisting map if and only if $\tau\xcirc s\xcirc\tau$ is
also, where $\tau$ denotes the flip.

\begin{theorem}\label{teorema 4.1} Let $A$ be an algebra and $s\colon k[t]/\langle t^2
\rangle \ot A\to A\ot k[t]/\langle t^2 \rangle$ a twisting map. The maps $\iota_0\colon A\to A$
and  $\iota_1\colon A\to A$, defined by
\begin{equation}
s(t\ot a)= \iota_0(a) \ot 1 + \iota_1(a)\ot t,\label{eq2}
\end{equation}
satisfy:

\begin{enumerate}

\smallskip

\item $\iota_1$ is a morphism of algebras.

\smallskip

\item $\iota_0(ab)= \iota_0(a)b + \iota_1(a) \iota_0(b)$ (that is, $\iota_0$ is an
$\iota_1$-derivation).

\smallskip

\item $\iota_0^2=0$ and $\iota_0\xcirc \iota_1=-\iota_1\xcirc \iota_0$.

\smallskip

\end{enumerate}
Conversely, given maps $\iota_0\colon A\to A$ and  $\iota_1\colon A\to A$ satisfying (1)--(3),
the formula~\eqref{eq2} determines a twisting map.

\end{theorem}

\begin{proof} Left to the reader.
\end{proof}

\begin{lemma}\label{combinatorio} We define
$$
A_n^N=\sum_{k=0}^{N-n}(-1)^k\binom{n+k}n
$$
Then the following facts hold:
\begin{align}
& 2A_n^N-A_{n-1}^{N-1} =(-1)^{N-n} \binom  Nn,\label{eq3} \\
& 2A_n^N-A_{n-1}^{N} =(-1)^{N-n} \binom  {N+1}n,\label{eq4} \\
& 2A^N_0=(-1)^N+1,\qquad  A_N^N=1.\label{eq5}
\end{align}

\end{lemma}

\begin{proof} Assume $N-n=2j$ is even. Since $\binom{n+k}n-\binom{n+k+1}n =-\binom{n+k}{n-1}$,
$$
A_n^N = - \sum_{k=0}^{j-1}\binom{n+2k}{n-1}+ \binom  Nn\quad\text{and}\quad A_n^N=
\sum_{k=0}^j\binom{n+2k-1}{n-1}.
$$
Summing both results, we obtain
$$
2A_n^N=\sum_{k=0}^{N-n}(-1)^k\binom{n-1+k}{n-1}+(-1)^{N-n} \binom{N}{n},
$$
and so
$$
2A_n^N=A_{n-1}^{N-1}+(-1)^{N-n}\binom{N}{n},
$$
which is~\eqref{eq3}. The case $N-n$ odd is similar. The equality~\eqref{eq4} follows
from~\eqref{eq3}. Finally,~\eqref{eq5} can be easily checked by a direct computation.
\end{proof}

\begin{lemma}\label{sistemasequivalentes} A vector $y=(y_i)\in k^m$ satisfies the set of
equalities
$$
\sum_{i=h}^{m-1}y_{i+1}A_{i-h}^i=0 \qquad for\ all \ h=0,\dots, m-1,
$$
where $A^N_n=\sum_{k=0}^{N-n}(-1)^k\binom{n+k}{n}$, if and only if it satisfies the set of
equalities:
$$
\sum_{i=h}^m \binom ih y_i=(-1)^h y_h\qquad for\ all \ h= 0,\dots, m.
$$
\end{lemma}

\begin{proof}
Set
$$
A(h) = \sum_{i=h}^{m-1}y_{i+1}A_{i-h}^i\quad\text{and}\quad B(h)=\sum_{i=h}^m \binom ih
y_i-(-1)^h y_h.
$$
We claim that
$$
2A(h)-A(h+1)=(-1)^h B(h+1)\qquad for\ all\ h=0,\dots, m-1.
$$
From this equality and from $A(0)=B(0)$ it follows by induction on $h$ that
$$
A(h)=\sum_{k=0}^h(-1)^k2^{h-k}B(k),
$$
and hence the lemma. Now we prove the claim:
\begin{eqnarray*}
2A(h)-A(h+1) &=& \sum_{i=h}^{m-1}2A_{i-h}^iy_{i+1}- \sum_{i=h+1}^{m-1}A_{i-h-1}^iy_{i+1}\\
&=& \sum_{i=h+1}^{m-1}(2A_{i-h}^i-A_{i-h-1}^i)y_{i+1} + y_{h+1}2A_0^h\\
&=&  \sum_{i=h+1}^{m-1}(-1)^h \binom{i+1}{i-h}y_{i+1} + y_{h+1}((-1)^h+1) \\
&=&  \sum_{i=h+1}^m(-1)^h \binom{i}{h+1}y_i + y_{h+1} \\
&=& (-1)^h B(h+1),
\end{eqnarray*}
where the third equality follows from equalities~\eqref{eq4} and \eqref{eq5}.
\end{proof}

\begin{theorem}\label{teorema 4.2} Let $s\colon k[t]/\langle t^2 \rangle \ot k[Y]\to
k[Y]\ot k[t]/\langle t^2 \rangle$ be a twisting map and let $\iota_j\colon k[Y]\to k[Y]$
($j=0,1$) be the maps introduced in Theorem~\ref{teorema 4.1}. Write $\iota_0(Y)=Q$ and
$\iota_1(Y)=P$. If $Q = \sum_{i=0}^m q_i Y^i\neq 0$, then $P=-Y+p_0$ and

\begin{enumerate}

\smallskip

\item If $p_0 = 0$, then $q_i = 0$ for $i$ odd,

\smallskip

\item If $p_0 \ne 0$, then $\displaystyle{\sum_{j=i}^m \binom ji q_j p_0^{j-i}=(-1)^i
q_i}$ for all $i=0,\dots, m$.

\smallskip
\end{enumerate}
(Note that item~(2) implies that $m$ is even). Conversely, if $Q=0$ and $P$ is arbitrary or $P
=-Y+p_0$ and conditions~(1), (2) are satisfied, then there is a unique twisting map $s\colon
k[t]/\langle t^2 \rangle \ot k[Y]\to k[Y]\ot k[t]/\langle t^2 \rangle$ such that $\iota_0(Y)=Q$
and $\iota_1(Y)=P$, where the maps $\iota_0$ and $\iota_1$ are defined as in
Theorem~\ref{teorema 4.1}.
\end{theorem}

\begin{proof} Let $s$ be a twisting map. Assume $Q\neq 0$. In the sequel we adopt the
convention that $P^0=1$ even if $P=0$. It is easy to check by induction on $l$, that
\begin{equation}
\iota_0(Y^l)= \sum_{i=0}^{l-1} \iota_1(Y)^i\iota_0(Y)Y^{l-i-1}= Q\sum_{i=0}^{l-1}P^iY^{l-i-1}.
\label{eq7}
\end{equation}
We claim that $\dg(P) = 1$. Suppose $\dg(P) \ne 1$ or $P=0$. Let $l\ge 1$. If $\dg(P)>1$, then
$\iota_0(Y^l)$ has degree $(l-1)\dg(P)+\dg(Q)$, and if $P$ is a constant, then it has degree
$l-1+\dg(Q)$. In  both cases it is easy to see that $\iota_0^2=0 \Leftrightarrow \dg(Q)=0$. Let
$Q=q\in k\setminus\{0\}$. It is immediate that if $\dg(P)>1$, then
$\iota_0(\iota_1(Y))=\iota_0(P)$ has degree $(\dg(P) - 1)\dg(P)$. Since
$\iota_1(\iota_0(Y))=q$, it is impossible that $\iota_0 \xcirc\iota_1 = -\iota_1\xcirc\iota_0$.
So $P=0$ or $\dg(P) \le 1$. If $P$ is a constant, then $\iota_0(\iota_1(Y))=0$ and also in this
case $\iota_0\xcirc\iota_1\ne-\iota_1\xcirc\iota_0$. This proves the claim. Write $P =
p_1Y+p_0$. We assert that the following facts hold

\begin{enumerate}

\smallskip

\item[(3)] $\displaystyle{\iota_0^2(Y) = Q\sum_{h=0}^{m-1}\sum_{i=h}^{m-1}
\sum_{j=i-h}^i\binom{j}{i-h}p_1^{j-i+h} p_0^{i-h} q_{i+1} Y^h}$.

\smallskip

\item[(4)] $\displaystyle{\iota_1(\iota_0(Y))=\sum_{i=0}^m\sum_{j=i}^{m}\binom{j}{i}
q_j p_0^{j-i} p_1^i  Y^i}$ and $\displaystyle{\iota_0(\iota_1(Y))=p_1\sum_{i=0}^m q_iY^i}$.

\smallskip

\end{enumerate}
In fact by~\eqref{eq7}, we have
$$
\iota_0^2(Y) = \sum_{i=0}^m q_i\iota_0(Y^i)=Q\sum_{i=1}^m\sum_{j=0}^{i-1} q_i P^jY^{i-j-1}.
$$
Since $P^j=\sum_{l=0}^j\binom{j}{l}p_1^{j-l}p_0^lY^{j-l}$, this gives
$$
\iota_0^2(Y) = Q\sum_{i=1}^m\sum_{j=0}^{i-1}\sum_{l=0}^jq_i \binom{j}{l}
p_1^{j-l}p_0^lY^{i-l-1}.
$$
A change of indices yields item~(3). Next we check item~(4). We have
$$
\iota_1(\iota_0(Y))=\sum_{j=0}^mq_j\iota_1(Y)^j= \sum_{j=0}^m \sum_{l=0}^j q_j\binom{j}{l}
p_1^{j-l}p_0^lY^{j-l}=\sum_{j=0}^m\sum_{i=0}^j q_j\binom{j}{i} p_1^ip_0^{j-1}Y^i.
$$
Interchanging the sums we obtain item~(4), since the second equality is clear. Considering now
the terms of maximal degree in items~(3) and~(4), we get
$$
\sum_{j=0}^{m-1}p_1^j=0\qquad\text{and}\qquad p_1^m=-p_1,
$$
since, by Theorem~\ref{teorema 4.1} we know that
\begin{equation}
\iota_0^2(Y)=0\qquad\text{and}\qquad \iota_1(\iota_0(Y))= -\iota_0(\iota_1(Y)).\label{pepe}
\end{equation}
Hence $p_1=-1$ and equalities~(3) and~(4) become

\begin{enumerate}

\smallskip

\item[(5)] $\displaystyle{\iota_0^2(Y) = Q\sum_{h=0}^{m-1}\sum_{i=h}^{m-1}
\sum_{j=i-h}^i \binom{j}{i-h} (-1)^{j-i+h} p_0^{i-h} q_{i+1} Y^h}$.

\smallskip

\item[(6)] $\displaystyle{\iota_1(\iota_0(Y))=\sum_{i=0}^m\sum_{j=i}^{m}\binom{j}{i}
q_j p_0^{j-i} (-1)^i  Y^i}$ and $\displaystyle{\iota_0(\iota_1(Y))= -\sum_{i=0}^m q_iY^i}$.

\smallskip

\end{enumerate}
From this it follows immediately that equalities~\eqref{pepe} implies items~(1) and (2). Now we
prove the second part. First note that for $P,Q\in k[Y]$ arbitrary, equalities $\iota_0(Y) = Q$
and $\iota_1(Y) = P$ determine unique maps $\iota_0,\iota_1\colon k[Y]\to k[Y]$ satisfying
conditions~(1) and~(2) of Theorem~\ref{teorema 4.1}. Clearly condition~(3) is also fulfilled if
and only if $\iota_0^2(Y)=0$ and $\iota_1(\iota_0(Y))=-\iota_0(\iota_1(Y))$. When $Q=0$ these
last equalities are trivially true.  Assume now $Q \neq 0$ and $P = -Y + p_0$. Then, arguing as
above, we obtain that~(5) and~(6) are satisfied. From this it follows immediately that, if $p_0
= 0$ and $q_i = 0$ for $i$ odd, then equalities~\eqref{pepe} are true. Assume now that $p_0\ne
0$. In this case from~(5) and~(6) it follows that $\iota_0^2(Y)=0$ is equivalent to the fact
that the $q_i$'s satisfy
$$
\sum_{i=h}^{m-1} q_{i+1}p_0^{i+1} A_{i-h}^i=0\qquad \text{for all $h= 0,\dots,m-1$,}
$$
where $A^N_n:=\sum_{k=0}^{N-n}(-1)^k\binom{n+k}{n}$ and that $\iota_1(\iota_0(Y))=
-\iota_0(\iota_1(Y))$ is equivalent to the fact that the $q_i$'s satisfy
$$
\sum_{j=i}^m \binom ji q_jp_0^{j-i}=(-1)^iq_i\qquad\text{for all $i= 0,\dots, m$,}
$$
which is true by item~(2). But, by Lemma~\ref{sistemasequivalentes}, applied to
$\{y_i=q_ip_0^i\}$, the last set of equalities implies the first one. So the theorem is proved.
\end{proof}

In the previous theorem we found necessary and sufficient conditions, on polynomials $P,Q\in
k[Y]$, in order that a twisting map
$$
s\colon k[t]/\langle t^2\rangle\ot k[Y]\to k[Y]\ot k[t]/\langle t^2\rangle
$$
such that $s(t\ot Y) = P\ot t + Q\ot 1$ exists. If $Q=0$, then $P$ is arbitrary and if $Q\ne 0$
then $P = -Y+p_0$ and items~(1) or~(2) of Theorem~\ref{teorema 4.2} must be satisfied,
depending on if $p_0 = 0$ or $p_0\ne 0$. In the first case the condition is simply that $Q\in
k[Y^2]$. The second case is more involved and we give a complete solution under the hypothesis
that $k$ is a characteristic zero field.

\begin{corollary} (Classification of the non-commutative extension of the algebra of
dual numbers by $k[Y]$). Let $k$ be a characteristic zero field. If $Q\ne 0$, then any choice
of $p_0\in k\setminus\{0\}$, $m$ even and $q_0,q_2,\dots, q_m\in k$, with $q_m\ne 0$ determines
univocally polynomials $P=-Y+p_0$ and $Q=\sum_{i=0}^m q_i Y^i$ satisfying condition~(2) of
Theorem~\ref{teorema 4.2}.
\end{corollary}

\begin{proof} For each $i\ge 0$, let $y_i = q_ip_0^i$. Consider
$\mathbf{y}=(y_0,\dots,y_m)$ as a column vector. By item~(3) of Theorem~\ref{teorema 4.2},
there is a twisting map
$$
s\colon k[t]/\langle t^2\rangle\ot k[Y]\to k[Y]\ot k[t]/\langle t^2\rangle
$$
such that $s(t\ot Y) = P\ot t + Q\ot 1$, if and only if $\sum_{j=i}^m \binom ji y_j-(-1)^i
y_i=0$ for $i=0,\dots, m$. Write the system of equations
$$
B(i)=\sum_{j=i}^m \binom ji y_j-(-1)^i y_i=0\qquad (j=0,\dots,m)
$$
in the matrix form $\mathbf{C} \mathbf{y} = \mathbf{0}$. It is then easy to see that
$\mathbf{C} = (c_{ij})_{{}_{0\le i,j\le m}}$ is
$$
\begin{pmatrix}
0&1&1&1&1&1&1& \hdotsfor{2}&  1 \\[3pt]
0&2&2&3&4&5&6& \hdotsfor{2}& \binom m1\\[3pt]
0&0&0&3&6&10&15 &\hdotsfor{2}& \binom m2\\[3pt]
0&0&0&2&4&10&20 &\hdotsfor{2}& \binom m3\\[3pt]
0&0&0&0&0&5&15 &\hdotsfor{2}& \binom m4\\[3pt]
0&0&0&0&0&2&6 &\hdotsfor{2}& \binom m5\\[3pt]
\vdots&\vdots&\vdots & \vdots&\vdots&\vdots&\vdots&\ddots &&\vdots\\[3pt]
0&0&0&0&0&0&0 &\cdots& \binom{m-1}1&\binom m2\\[3pt]
0&0&0&0&0&0&0 &\cdots& 2&\binom m1
\end{pmatrix}
$$
By the shape of this matrix it is clear that the even rows are linearly independent and so we
only need to prove that $\rank(\mathbf{C}) = m/2$. For this it suffices to check that
\begin{equation}
\sum_{k=0}^n (-1)^k c_{i,2n-k}\binom{n}{k}=0\qquad\text{for $n\ge 1$ and $i=0,\dots,m$,}
\label{eq8}
\end{equation}
since then the even columns will be linear combinations of the previous ones. Let
$$
D_{ij}^{(n)} = \begin{cases} \binom{j-n}{i-n} & \text{if $n\le i\le j$,}\\ 0 &
\text{otherwise,}\end{cases}\quad\text{and}\quad E_{ij}^{(n)} = \begin{cases} (-1)^{j+1}
\binom{n}{j-i} & \text{if $j-n\le i\le j$,}\\ 0 & \text{otherwise.}\end{cases}
$$
Since $E_{i,2n}^{(n)} + D_{i,2n}^{(n)} = 0$ in order to prove~\eqref{eq8} it is enough to show
that
$$
E_{ij}^{(n)} + D_{ij}^{(n)} = \sum_{k=0}^n \binom{n}{k} (-1)^k c_{i,j-k}\quad\text{for
$i=0,\dots,m$ and $j\ge n$.}
$$
This follows immediately from the equalities
\begin{align*}
& E^{(0)}_{ij}+ D^{(0)}_{ij} = c_{ij},\\
&D^{(n)}_{ij}=\sum_{k=0}^n (-1)^k D^{(0)}_{i,j-k}\binom nk,\\
& E^{(n)}_{hj}=\sum_{k=0}^n (-1)^k E^{(0)}_{i,j-k}\binom nk,
\end{align*}
for $i=0,\dots,m$ and $j\ge n$. The first and the third one can be checked by a direct
computation, while the second one by induction on $n$.
\end{proof}

Let $s$, $P$ and $Q$ be as in Theorem~\ref{teorema 4.2}. Let $\alpha_j\colon k[t]/\langle
t^2\rangle\to k[t]/\langle t^2\rangle$ be the maps defined by
$$
\tau\xcirc s \xcirc \tau (Y\ot t)= \sum_{j=0}^{\infty} \alpha_j(t)\ot Y^j,
$$
where $\tau$ is the flip. If $Q=0$ and $P=\sum_{i=0}^n p_iY^i$, then $\alpha_j(t)=p_jt$. If $P$
and $Q$ are as in items~(1) and (2) of Theorems~\ref{teorema 4.2}, then
$$
\alpha_j(t)=\begin{cases}
                         q_0+p_0t &\quad\text{if $j=0$,}\\
                         q_1-p_1t &\quad\text{if $j=1$,}\\
                         q_j      &\quad\text{if $2\le j\le m$,}\\
                         0        &\quad\text{if $j>m$.}\\
            \end{cases}
$$

\section{Twisted extensions by power series}
\setcounter{equation}{0}
Let $k$ be a commutative ring. This section is devoted to the study of twisting tensor products
between the power series ring $k[[Y]]$ and a filtrated complete algebra $A$. Hence we work in
the monoidal category $\CMod$ of complete filtrated $k$-modules (see Section~1). Recall that
the tensor product of $\CMod$ is denoted by $\hat{\ot}$. We will use freely the notations
introduced in Sections~1 and~2.

\begin{lemma}\label{lemma 5.1} Let $A$ be a filtrated complete algebra and $(\alpha_j\colon
A\to A)_{j\ge 0}$ a family of continuous maps. If for each $i\ge 0$ there exists $n_0\ge 0$
such that $\alpha_0^n(A)\subseteq A_i$ for all $n\ge n_0$, then for each $i,j\ge 0$ there
exists $r_0\ge 0$ such that $\gamma^{(r)}_j(A_h)\subseteq A_i$ whenever $r+h\ge r_0$.
\end{lemma}

\begin{proof} We proceed by induction on $j$. First we assume $j=0$. By hypothesis there
exists $n_0\ge 0$, such that $\gamma^{(n)}_0(A) = \alpha_0^n(A)\subseteq A_i$ whenever $n\ge
n_0$. Since $\alpha_0$ is continuous, there exists $h_0\ge 0$ such that $\gamma^{(n)}_0(A_h) =
\alpha_0^n(A_h)\subseteq A_i$, for each $n<n_0$ and $h>h_0$. Clearly we can take $r_0 =
n_0+h_0$. Assume the lemma is valid for $j$ and write
$$
\gamma^{(r)}_{j+1} = \sum_{l_1 = 0}^{r-1}\sum_{l_2 = 1}^{j+1} \alpha_0^{l_1}\xcirc \alpha_{l_2}
\xcirc \gamma^{(r-l_1-1)}_{j+1-l_2}.
$$
Since $\alpha_0^n(A)\subseteq A_i$ for all $n\ge n_0$ in order to complete the inductive step
it suffices to show that for all $l_1 < n_0$ and $l_2\le j+1$, there exists $r_0\ge 0$ such
that
$$
\alpha_0^{l_1}\xcirc \alpha_{l_2}\xcirc \gamma^{(r-l_1-1)}_{j+1-l_2}(A_h)\subseteq A_i \quad
\text{whenever $r+h\ge r_0$,}
$$
which follows immediately from the continuity of $\alpha_0^{l_1}\xcirc \alpha_{l_2}$ and the
inductive hypothesis.
\end{proof}

\begin{remark}\label{remark 5.2} If $s\colon k[[Y]]\hat{\ot} A\to A\hat{\ot} k[[Y]]$ is a
twisting map, then $\ker(\alpha_0)$ is a closed subalgebra of $A$. Moreover,

\begin{itemize}

\smallskip

\item If $\alpha_0=0$, then $\alpha_1$ is a endomorphism of algebras.

\smallskip

\item Let $\nu>1$. If $\alpha_j=0$ for $1<j<\nu$, then
$$
\alpha_{\nu}(ab)= \alpha_1(a)\alpha_{\nu}(b) + \alpha_{\nu}(a)\alpha_1^{\nu}(b).
$$

\end{itemize}

\end{remark}

\begin{theorem}\label{teorema 5.3} Let $A$ be a filtrated complete algebra and $s\colon k[[Y]]
\hat{\ot} A\to A\hat{\ot} k[[Y]]$ a twisting map. The equation
$$
s(Y\hat{\ot} a) = \sum_{j=0}^{\infty} \alpha_j(a)\hat{\ot} Y^j,\label{eq9}
$$
defines a family of maps $\alpha_j\colon A\to A$, which satisfies:

\begin{enumerate}

\smallskip

\item The $\alpha_j$'s are continuous maps.

\smallskip

\item For each $i\ge 0$ there exists $n_0$ such that $\alpha_0^n(A)\subseteq A_i$ for all
$n\ge n_0$.

\smallskip

\item $\alpha_j(1) = \delta_{j1}$, where $\delta_{j1}$ denotes the symbol of Kronecker.

\smallskip

\item For all $j\ge 0$ and all $a,b\in A$,
$$
\quad\qquad\alpha_j(ab) = \sum_{r=0}^{\infty}\alpha_r(a)\gamma_j^{(r)}(b)\quad\text{(this
formula makes sense by Lemma~\ref{lemma 5.1}).}
$$

\smallskip

\end{enumerate}
Moreover,
\begin{equation}
s\left(\sum_{r=0}^{\infty} Y^r\hat{\ot} a_r\right )= \sum_{j=0}^{\infty}\left(
\sum_{r=0}^{\infty}\gamma_j^{(r)}(a_r) \right)\hat{\ot} Y^j.\label{eq10}
\end{equation}
Conversely, given maps $\alpha_j\colon A\to A$ ($j\ge 0$) satisfying (1)--(4), the
formula~\eqref{eq10} defines a twisting map.

\end{theorem}

\begin{proof} Items~(1) and (2) follows easily from the continuity of $s$, and items~(3)
and~(4) can be checked as in the proof of Theorem~\ref{teorema 2.1}. To check~\eqref{eq10} we
can assume that only one $a_r\ne 0$. In this case we can proceed again as in the proof of
Theorem~\ref{teorema 2.1}. Conversely, assume we have a family of continuous maps
$(\alpha_j)_{j\ge 0}$ satisfying (1), (2), (3) and (4) and define $s$ by the
formula~\eqref{eq10}. By Lemma~\ref{lemma 5.1} this map is well defined and it is continuous.
We leave the task to prove that $s$ is a twisting map to the reader.
\end{proof}

\begin{theorem}\label{teorema 5.4} Let $\alpha\colon A \to A$ be an automorphism of filtrated
completed algebras. For each $i\ge 1$, let $\beta_i\colon A\to A$ be a continuous
$(\alpha,\alpha^{i+1})$-derivation. If
$$
\alpha^r(\beta_i(a))\beta_{i'}(b)=0 \quad\text{for all $r\in\bZ$ and $a,b\in A$ whenever
$i+i'\ge 3$,}
$$
then, the formula
$$
s(Y\ot a)=\sum_{j=0}^{\infty} \alpha_j(a)\ot Y^j,
$$
where the maps $\alpha_j \colon A\to A$ ($j\ge 0$) are constructed as in Lemma~\ref{lema 2.5},
defines a twisting map $s\colon k[[Y]]\hat{\ot} A\to A\hat{\ot} k[[Y]]$.
\end{theorem}

\begin{proof} Mimic the proof of Theorem~\ref{teorema 2.7}.
\end{proof}

\begin{lemma}\label{lema 5.5} Let $\sum_{ij} a_{ij} X^i\hat{\ot} Y^j\in k[[X]]\hat{\ot}
k[[Y]]$. Assume that $a_{i0} =0$ for all $i$. The equality
$$
\alpha_1(X)=  \sum_{i\ge 0} a_{i1}X^i
$$
defines a continuous algebra map $\alpha_1\colon k[[X]]\to k[[X]]$ if and only if the
independent term $a_{01}$ of $\alpha_1(X)$ is nilpotent. Moreover, in this case, there is a
unique family continous maps $\bigl(\alpha_j\colon k[[X]]\to k[[X]]\bigr)_{j\ge 2}$, that
satisfy
\begin{align*}
&\alpha_{j+1}(1)=\delta_{j+1,1},\\
&\alpha_{j+1}(X) = \sum_{i\ge 0} a_{i,j+1}X^i,\\
&\alpha_{j+1}(X^{n+1}) = \sum_{r=1}^{j+1} \alpha_r(X^n)\gamma_{j+1}^{(r)}(X),
\end{align*}
for all $j\ge 1$.
\end{lemma}

\begin{proof} The first assertion it is immediate. In order to prove the second one it will be
sufficient to show that
$$
\alpha_j(X^n)\in \sum_{r=0}^{n-j+1} a_{01}^{n-j-r+1} X^r k[[X]]\qquad \text{if $n\ge j$.}
$$
We will prove this fact by induction on $j$. For $j=1$ this is clear. Assume it is true for $j$
and for $\alpha_{j+1}(X^h)$ with $j\le h\le n$. Then, by the inductive hypothesis and the facts
that,
\begin{align*}
& \alpha_{j+1}(X^{n+1}) = \sum_{s=1}^{j+1} \alpha_s(X^n)\gamma_{j+1}^{(s)}(X)\\
\intertext{and}
&\gamma_{j+1}^{(j+1)}(X) = \alpha_1^{j+1}(X)\in a_{01}k[[X]]+Xk[[X]],
\end{align*}
we get that for all $n\ge j$,
$$
\alpha_{j+1}(X^{n+1}) \in \sum_{s=1}^j \sum_{r=0}^{n-s+1}\! a_{01}^{n-s-r+1} X^r k[[X]] +
a_{01}^{n-j-r} X^r k[[X]](a_{01}k[[X]]+Xk[[X]]).
$$
Using this the proof can be easily finished.
\end{proof}

\begin{theorem}\label{teorema 5.6} Let $\sum_{ij} a_{ij} X^i\hat{\ot} Y^j\in k[[X]]\hat{\ot}
k[[Y]]$. If $a_{i0} =0$ for all $i$ and $a_{01}$ is nilpotent, then there is a unique twisting
map $s\colon k[[Y]]\hat{\ot} k[[X]]\to k[[X]]\hat{\ot} k[[Y]]$ satisfying
$$
s(Y\hat{\ot} X) = \sum_{ij} a_{ij} X^i\hat{\ot} Y^j
$$
Moreover
$$
s(Y\hat{\ot} P) = \sum_{j\ge 1}\alpha_j(P)\hat{\ot} Y^j,
$$
where $\alpha_j\colon k[[X]]\to k[[X]]$ ($j\ge 1$) are the maps introduced in Lemma~\ref{lema
5.5}.
\end{theorem}

\begin{proof} The uniqueness and the last assertion are immediate. Let us prove the existence.
Let $\alpha_0 = 0$.  By Lemma~\ref{lema 5.5} we know that the maps $\alpha_j$ are well defined
and continuous. Moreover, it is evident that items~(1), (2) and (3) of Theorem~\ref{teorema
5.3} are satisfied. So we only must prove item~(4), which (by linearity and continuity) reduce
to check that
$$
\alpha_j(X^mX^n) =  \sum_{l=1}^j \alpha_l(X^m)\gamma_j^{(l)}(X^n)\quad\text{for all $m,n\ge
0$.}
$$
For $j=1$ this follows from Remark~\ref{remark 5.2}. Assume that the result is true for
$\alpha_q$ with $q<j$ and for $\alpha_j(X^mX^b)$ with $b\le n$.  By the recursive definition of
$\alpha_j$ and the induction hypothesis,
\begin{align*}
\alpha_j(X^mX^{n+1}) &= \sum_{r=1}^j\alpha_r(X^mX^n)\gamma_j^{(r)}(X)\\
&= \sum_{r=1}^j\sum_{l=1}^r \alpha_l(X^m)\gamma_r^{(l)}(X^n)\gamma_j^{(r)}(X)\\
&= \sum_{l=1}^j\sum_{r=l}^j \alpha_l(X^m)\gamma_r^{(l)}(X^n)\gamma_j^{(r)}(X).
\end{align*}
So it is enough to show that
$$
\gamma_j^{(l)}(X^{n+1}) = \sum_{r=l}^j\gamma_r^{(l)}(X^n)\gamma_j^{(r)}(X)\quad\text{for }
l=1,\dots,j.
$$
We prove this formula by induction on $l$. When $l=1$ this is true by the recursive definition
of $\alpha_j(X^{n+1}) = \gamma_j^{(1)}(X^{n+1})$. Suppose $l>1$. Then, we have
\allowdisplaybreaks
\begin{align*}
\gamma_j^{(l)}(X^{n+1}) & = \sum_{q_1=1}^{j-l+1}\sum_{|q_2,\dots,q_l|=j-q_1}
\alpha_{q_1}\bigl(\alpha_{q_2\dots q_l}(X^{n+1})\bigr)\\
&= \sum_{q=1}^{j-l+1}  \alpha_q\bigl(\gamma^{(l-1)}_{j-q}(X^{n+1})\bigr)\\
&=\sum_{q=1}^{j-l+1}\sum_{s=l-1}^{j-q}\alpha_q\bigl(\gamma_s^{(l-1)}(X^n)
\gamma_{j-q}^{(s)}(X)\bigr)\\
& =\sum_{s=l-1}^{j-1}\sum_{q=1}^{j-s}\sum_{h=1}^q \alpha_h\bigl(\gamma_s^{(l-1)}(X^n)\bigr)
\gamma_q^{(h)}\bigl(\gamma_{j-q}^{(s)}(X)\bigr)\\
& =\sum_{s=l-1}^{j-1}\sum_{h=1}^{j-s}\sum_{q=h}^{j-s} \alpha_h\bigl(\gamma_s^{(l-1)}(X^n)\bigr)
\gamma_q^{(h)}\bigl(\gamma_{j-q}^{(s)}(X)\bigr)\\
& = \sum_{r=l}^j\gamma_r^{(l)}(X^n)\gamma_j^{(r)}(X),
\end{align*}
as desired.
\end{proof}

\begin{remark}\label{remark 5.7} Let $\sum_{ij} a_{ij} X^i\hat{\ot} Y^j\in k[[X]]\hat{\ot}
k[[Y]]$. Suppose that $a_{0j} =0$ for all $j$ and $a_{10}$ is nilpotent. By
Theorem~\ref{teorema 5.6} we know that there exists a unique twisting map $s\colon
k[[X]]\hat{\ot} k[[Y]]\to k[[Y]]\hat{\ot} k[[X]]$ satisfying
$$
s(X\hat{\ot} Y) = \sum_{ij} a_{ij} Y^j\hat{\ot} X^i.
$$
But then the map  $\tau\xcirc s\xcirc \tau$, where  $\tau \colon k[[Y]]\hat{\ot} k[[X]]\to
k[[X]]\hat{\ot} k[[Y]]$ is the flip, is a twisting map taking $Y\ot X$ to $\sum_{ij} a_{ij}
X^i\hat{\ot} Y^j$.
\end{remark}

\appendix

\section{}
\setcounter{equation}{0}
This appendix is devoted to prove Theorem~\ref{teorema 3.4}. So, we assume that $k$ is a
commutative domain. Since $s\colon k[Y]\ot k[X]\to k[X]\ot k[Y]$ is a twisting map there is a
family of maps $(\alpha_n\colon k[X]\to k[X])_{n\ge 0}$ satisfying the conditions established
in Theorem~\ref{teorema 2.1}. An easy argument (see the proof of Theorem~\ref{teorema 3.4} at
the end of this appendix) shows that, under the hypothesis of item~(1) of Theorem~\ref{teorema
3.4}, $\alpha_0 = 0$, $\alpha_1=ev_0$ and there exists $\nu\ge 2$ such that $\alpha_\nu\ne 0$
and $\alpha_k=0$ for $1<k<\nu$. We do not will use these facts until Lemma~\ref{alfasygammas}.
\begin{lemma}\label{nuevo} Let $j_0,r_0\ge 0$. If $\alpha_l(X^{j_0}) = 0$ for all $l\le r_0$,
then $\alpha_l(X^j) = 0$ for all $l\le r_0$ and $j\ge j_0$.
\end{lemma}

\begin{proof} By item~(3) of Theorem~\ref{teorema 2.1},
$$
\alpha_l(X^j)=\sum_{h=0}^l \alpha_h (X^{j_0})\gamma^{(h)}_l(X^{j-j_0}) = 0,
$$
for all $l\le r_0$ and $j>j_0$, as desired.
\end{proof}

\begin{lemma}\label{diagonalestriviales} Let $j_0,\rho_0\ge 0$ and $a\ge 1$ be integers. If
\begin{enumerate}

\smallskip

\item $\alpha_l(X^{j_0})=0$ for all $l\le \rho_0$,

\smallskip

\item $\gamma_s^{(l)}(X)=0$ for all $l>\rho_0$ and $s<l+a$,

\smallskip

\end{enumerate}
then $\alpha_l(X^{j+i})=0$ for all $i\ge 0$, $l\le \rho_0+ia$ and $j\ge j_0$.
\end{lemma}

\begin{proof} Note that the case $i=0$ follows immediately from Lemma~\ref{nuevo}.
We now prove the assertion for $i=1$. Take $j\ge j_0$ and $l\le \rho_0+a$. By item~(2) we have
$\gamma^{(h)}_l(X) = 0$ for all $h>\rho_0$, since $l<\rho_0+h$. So,
$$
\alpha_l(X^{j+1}) = \sum_{h=0}^{\infty}\alpha_h(X^j)\gamma_l^{(h)}(X) =
\sum_{h=0}^{\rho_0}\alpha_h(X^j) \gamma_l^{(h)}(X) = 0,
$$
where the last equality follows from the case $i=0$. An easy induction argument on $i$
concludes the proof.
\end{proof}

\begin{lemma}\label{diagonales} Let $j_0,\rho_0\ge 0$ and $a\ge 1$ be integers. Assume that
\begin{enumerate}

\smallskip

\item $\alpha_{\rho_0+1}(X^{j_0})\ne 0$,

\smallskip

\item $\alpha_l(X^{j_0})=0$ for all $l\le \rho_0$,

\smallskip

\item $\gamma_s^{(l)} = \delta_{ls}ev_0$ for all $l> \rho_0$ and $s<l+a$,

\smallskip

\item $\gamma_{l+a}^{(l)} = \alpha_1\circ \alpha_{a+1}$ for all $l>\rho_0$.

\smallskip

\end{enumerate}
Then
$$
\alpha_{\rho_0+ia+1}(X^{j_0+iu}) = \alpha_{\rho_0+1}(X^{j_0})\alpha_{1,a+1}(X^u)^i \qquad
\text{for all $u\ge 1$ and $i\ge 0$.}
$$
Moreover, if for each $u\ge 1$ there is an $i\ge 0$ such that $\alpha_{\rho_0+ia+1}(X^{j_0+iu})
= 0$, then
\begin{alignat*}{2}
& \alpha_{1,a+1}=0,\\
& \alpha_{\rho_0+ia+1}(X^{j_0+iu})=0 &&\qquad \text{for all $u\ge 1$ and $i\ge 1$,}\\
& \gamma_s^{(l)}=\delta_{ls}ev_0 &&\qquad \text{for all $l>\rho_0$ and $s\le l+a$.}
\end{alignat*}
\end{lemma}

\begin{proof} For $i=0$ the first assertion is trivial. Suppose it is true for~$i$.
By Lemma~\ref{diagonalestriviales},
$$
\alpha_l(X^{j_0+iu})=0\quad\text{for $l\le \rho_0+ia$,}
$$
since $j_0+(u-1)i\ge j_0$. Write $p = \rho_0+ia+1$. By item~(3), we know that
$\gamma_{p+a}^{(l)}(X^u)=0$ for all $l>p$. So,
\begin{align*}
\alpha_{p+a}(X^{j_0+(i+1)u})&=\sum_{l=0}^{\infty}\alpha_l(X^{j_0+iu})
\gamma_{p+a}^{(l)}(X^u)\\
&= \alpha_p(X^{j_0+iu})\gamma_{p+a}^{(p)}(X^u)\\
&= \alpha_{\rho_0+1}(X^{j_0})\alpha_{1,a+1}(X^u)^i\alpha_{1,a+1}(X^u)\\
&= \alpha_{\rho_0+1}(X^{j_0})\alpha_{1,a+1}(X^u)^{i+1},
\end{align*}
where the third equality follows from the inductive hypothesis and item~(4). It is now clear
that if for each $u\ge 1$ there is an $i$ such that $\alpha_{\rho_0+ia+1}(X^{j_0+iu}) = 0$,
then $\alpha_{1,a+1}(X^u) = 0$, for all $u\ge 1$. Combined this with item~(2) of
Theorem~\ref{teorema 2.1}, we obtain $\alpha_{1,a+1} = 0$. The remainder assertions follows now
easily.
\end{proof}

\begin{lemma}\label{numerico}
Let $b,n_1,\dots,n_l\ge \nu\ge 2$ be integers such that:
\begin{enumerate}

\smallskip

\item $\sum_{i=1}^ln_i\le (l-1)b+\nu$,

\smallskip

\item If $i\ge 2$ and $n_i<b$, then $n_{i-1}\ge b(\nu-1)+2$,

\smallskip

\item If $\nu=2$, then $n_{i-2}\ge 2b-1$ for all $i\ge 3$ with $n_i<b$ and $n_{i-1}\le 2b-2$.

\smallskip

\end{enumerate}
Then $n_1=\nu$ and $n_i=b$ for $i=2,\dots,l$.
\end{lemma}

\begin{proof} If $n_i\ge b$ for all $i$, then by item~(1),
$$
(l-1)b+\nu\ge \sum_{i=1}^ln_i\ge bl.
$$
Hence $b=n_1=\dots=n_l=\nu$. Thus we can assume that the set of the indices $i_1<i_2<\dots
<i_m$, such that $n_{i_j}<b$, is not empty. We claim that $m=1$. In the hope of reaching a
contradiction we assume $m\ge 2$. We consider first the case $\nu\ge 3$. This implies
$2+\nu+(\nu-1)b>2b$. Set $M:=\{i_2,\dots, i_m\}$, write $M-1:=\{i-1:i\in M\}$ and set
$L:=\{i_1\}\cup M\cup (M-1)$. From item~(2) it follows easily that $i_2>2$ and $\{i_1\}$, $M$
and $(M-1)$ are pairwise disjoint sets, which implies $\# L = 2m-1$. Again by item~(2), we have
$n_{i-1}\ge b(\nu-1)+2$, for $i\in M$. Thus,
\begin{align*}
\sum_{i\in L} n_i &= n_{i_1}+ \sum_{i\in M}n_i + \sum_{i\in M-1 }n_i \\
&\ge \nu + (m-1)\nu + (m-1)(b(\nu-1)+2)\\
&> \nu+(m-1)2b,
\end{align*}
since $(\nu-1)b\ge 2b$ and $m-1>0$. Moreover, $n_i\ge b$ for $i\notin L$, and so,
$$
\sum_{i=1}^l n_i = \sum_{i\in L}n_i + \sum_{i\notin L}n_i > \nu+(m-1)2b+(l-(2m-1))b=\nu+b(l-1),
$$
which contradicts item~(1). It remains to consider the case $\nu=2$. Set
\begin{align*}
& I:=\{i_j: j\ge 2\text{ and } b+2\le n_{i_j-1}\le 2b-2\}
\intertext{and}
& J:=\{i_j: j\ge 2\text{ and } n_{i_j-1}\ge 2b-1\}.
\end{align*}
By item~(2), we have $\{i_1<i_2<\dots <i_m\} = \{i_1\}\cup I\cup J$ and $i_2>2$. Write
$$
I-1:=\{i-1:i\in I\},\quad I-2:=\{i-2:i\in I\}\quad\text{and}\quad J-1:=\{i-1:i\in J\}
$$
and set
$$
\bar I:= I\cup (I-1) \cup (I-2),\quad \bar J:= J\cup (J-1)\quad\text{and}\quad K:=\{i_1\}\cup
\bar I\cup \bar J.
$$
We assert that $\{i_1\}$, $I$, $J$, $I-1$, $J-1$ and $I-2$ are pairwise disjoint. It is
immediate that for $\{i_1\}$, $I$ and $J$ this is true. By definition, if $i\in I-1$, then
$n_i\ge b+2$. Therefore $i\notin \{i_1,\dots,i_m\}$ and so $I-1$ is disjoint from $\{i_1\}\cup
I\cup J$. Similarly $J-1$ is disjoint from $\{i_1\}\cup I\cup J$. Moreover $(I-1)\cap (J-1) =
\emptyset$, since $I\cap J = \emptyset$ and similarly $(I-2)\cap (I-1) = \emptyset = (I-2)\cap
(J-1)$. It remains to check that
\begin{equation}
\{i_1\}\cap (I-2) = I \cap (I-2) = J \cap (I-2) = \emptyset.\label{eqZ}
\end{equation}
But by item~(3),
$$
i\in I-2\Rightarrow i+2\in I \Rightarrow n_{i+2}<b\text{ and }n_{i+1}\le 2b-2 \Rightarrow
n_i\ge 2b-1\ge b,
$$
from which \eqref{eqZ} follows immediately. This finishes the proof of the assertion. It is
immediate now that
$$
k=1+3k_1+2k_2,
$$
where $k=\# K$, $k_1=\# I$ and $k_2=\# J$. Moreover,
$$
n_i+n_{i-1}\ge 2+2b-1=2b+1\quad\text{for all $i\in J$}
$$
and by condition~(3),
$$
n_i+n_{i-1}+n_{i-2}\ge 2+b+2+2b-1=3b+3\quad\text{for all $i\in I$}.
$$
Thus,
$$
\sum_{i\in K}n_i\ge 2+(3b+3)k_1+(2b+1)k_2> (k-1)b+2,
$$
since $k_1>0$ or $k_2>0$ by assumption. Since $n_i\ge b$ for $i\notin K$, we obtain
$$
\sum_{i=1}^l n_i=\sum_{i\in K}n_i+\sum_{i\notin K}n_i> (k-1)b+2+(l-k)b=(l-1)b+2,
$$
which contradicts item~(1). Hence $m = 1$, but then
$$
(l-1)b+\nu\ge \sum_{i=1}^ln_i\ge (l-1)b+\nu,
$$
where the first inequality is item~(1). So $n_{i_1}=\nu$ and $n_i = b$ for $i\ne i_1$. By
item~(2) this is only possible if $i_1 = 1$.
\end{proof}

In the sequel we assume that $\alpha_0=0$, $\alpha_1=\ev_0$ and there exists $n>1$ such that
$\alpha_n\ne 0$. Let $\nu$ be the least number satisfying this property. By Remark~\ref{remark
2.2} we know that $\alpha_{\nu}(X)\ne 0$. We define
$$
r_j(b)=\frac{(\nu-1)(b^j-1)}{b-1}\quad\text{for $b\ge \nu$ and $j\ge 1$.}
$$

\begin{lemma}\label{alfasygammas} Let $k_1>0$ and $b\ge\nu$ be integers. Assume $\alpha_{1k}=
0$ for $k=2,\dots,k_1$.

\begin{enumerate}

\smallskip

\item If $l>k_1$, then
$$
\gamma_{l+k_1}^{(l)} = \alpha_1^{l-1}\circ \alpha_{k_1+1}.
$$

\smallskip

\item If $\nu=2$, $b\ge 4$ and $\alpha_{ij}=0$ for $2\le i\le b+1$ and $j< b$, then
$$
\gamma^{(l)}_{l+k_1} = \alpha_1^{l-1}\circ \alpha_{k_1+1},
$$
as long as $k_1<2b+8$ and $l\ge b+2$.

\smallskip

\item Suppose $\alpha_{ij}=0$ for $2\le i\le 1+(\nu-1)b$ and $j< b$. When $\nu=2$ also suppose
that $\alpha_{ijk}=0$ for $2\le i\le 2b-2$, $j\le 2b-2$ and $k<b$. Let $j_0\in \mathds{N}$
arbitrary and set $r_0 = r_{j_0}(b)$. We have
$$
\gamma^{(r_0+1)}_{r_0+1+(\nu-1)b^{j_0}} = \alpha_1^{r_0}\circ \alpha_{(\nu-1)b^{j_0}+1}
+\alpha_{\nu}\circ \alpha_b^{r_0}
$$
and
$$
\gamma^{(l)}_{l+k_1} = \alpha_1^{l-1}\circ \alpha_{k_1+1},
$$
as long as $k_1\le (\nu-1)b^{j_0}$ and $l>r_0$ with $(l,k_1)\ne (r_0+1,(\nu-1)b^{j_0})$.

\end{enumerate}

\end{lemma}

\begin{proof} By definition
$$
\gamma_{l+k_1}^{(l)}= \sum_{|n_1,\dots, n_l|=l+k_1}\alpha_{n_1\dots n_l}.
$$
Suppose $\alpha_{n_1\dots n_l}\ne 0$. Since $\alpha_0=0$, this implies $n_1,\dots,n_l\ge 1$.
Assume $n_i=1$ for some $i$ and let $i_0$ be the greatest index satisfying this property. Being
$\alpha_{k1} = 0$ for all $k\ne 1$ it must be $n_i = 1$ for all $i\le i_0$. Since $|n_1,\dots,
n_l| = l+k_1 > l$ and $\alpha_{1k} = 0$ for $k=2,\dots, k_1$ the only possible choice is
$n_i=1$ for $i=1,\dots,l-1$ and $n_l=k_1+1$. Hence,
$$
\gamma_{l+k_1}^{(l)}= \alpha_1^{l-1}\circ \alpha_{k_1+1} + \sum_{|n_1,\dots, n_l|=l+k_1\atop
n_1,\dots, n_l\ge \nu}\alpha_{n_1\dots n_l},
$$
since $\alpha_2 =\dots = \alpha_{\nu-1} = 0$. Item~(1) follows now immediately, because
$$
l>k_1\Rightarrow 2l>l+k_1 = \sum_{i=1}^ln_i \Rightarrow n_i< 2\text{ for some $i$.}
$$
In order to prove item~(2) we proceed by contradiction. Suppose there is $\alpha_{n_1\dots n_l}
\ne 0$ with $n_1,\dots,n_l\ge 2$. Let $i_1<i_2<\dots<i_m$ be the indices such that $n_{i_j}<
b$. Set $J = \{i_j:j\ge 2\}$ and $K = \{i_1\}\cup J\cup (J-1)$. By hypothesis $n_i\ge b+2$ for
all $i\in J-1$. Consequently $\{i_1\}$, $J$ and $J-1$ are pairwise disjoint. So $2m-1\le l$ and
\begin{align*}
l+k_1 & = n_{i_1} + \sum_{i\in J} n_i + \sum_{i\in J-1} n_i + \sum_{i\notin K} n_i\\
& \ge 2+2(m-1)+(b+2)(m-1) + b(l-2m+1).
\end{align*}
Hence $k_1\ge (4-b)m + l(b-1)-2$ and so,
\begin{alignat*}{2}
k_1 &\ge (4-b)\frac{l+1}{2} + l(b-1)-2\qquad &&\text{since $m\le (l+1)/2$ }\\
&= l(1+b/2)-b/2\\
&\ge (b+2)(1+b/2)- b/2 &&\text{since $l\ge b+2$}\\
&=  b(b-1)/2 + 2b + 2\\
&\ge  2b + 8 &&\text{since $b\ge 4$}.
\end{alignat*}
This contradicts the fact that $k_1<2b+8$. Next we prove item~(3). Assume that $k_1\le
(\nu-1)b^{j_0}$ and $l>r_0$. Suppose there exists $\alpha_{n_1\dots n_l} \ne 0$ with
$n_1,\dots,n_l\ge \nu$ and $|n_1,\dots, n_l|=l+k_1$. We will prove that $n_1,\dots,n_l$ fulfill
the conditions of Lemma~\ref{numerico}. The first one follows from the fact that
$$
l>r_0\Rightarrow l-1\ge \frac{(\nu-1)(b^{j_0}-1)}{b-1}\Rightarrow (l-1)(b-1) \ge
(\nu-1)(b^{j_0}-1),
$$
and so
\begin{equation}
(l-1)b+\nu\ge (\nu-1)b^{j_0}+l \ge k_1+l = \sum_{i=0}^l n_i,\label{eqA}
\end{equation}
since $k_1\le (\nu-1)b^{j_0}$. If $i\ge 2$ and $n_i< b$, then necessarily $n_{i-1}\ge (\nu-1)
b+2$, since $\alpha_{n_{i-1}n_i}=0$ for $2\le n_{i-1}< (\nu-1)b+2$ and $n_i< b$. So,
condition~(2) holds. Finally, if $\nu=2$, $i\ge 3$, $n_i<b$ and $n_{i-1}\le 2b-2$, then
necessarily $n_{i-2}\ge 2b-1$, since $\alpha_{n_{i-2}n_{i-1}n_i}=0$ otherwise. This establishes
condition~(3). From Lemma~\ref{numerico} it follows now that $n_1=\nu$ and $n_i=b$ for $i>1$.
Hence,
$$
l+k_1 = \sum_{i=1}^ln_i = (l-1)b+\nu,
$$
and so $k_1 = (l-1)b-l+\nu \ge (\nu-1) b^{j_0}$, where the inequality follows from~\eqref{eqA}.
Since, by hypothesis $k_1\le (\nu-1)b^{j_0}$, we conclude that
$$
k_1=(\nu-1)b^{j_0}\quad\text{and}\quad l = \frac{(\nu-1)b^{j_0}+b-\nu}{b-1} = r_0+1.
$$
This finishes the proof.
\end{proof}

\begin{lemma}\label{curvas} Suppose that for each $a\ge 1$ and $j,k\ge 0$ there exists $i\ge
1$ such that $\alpha_{k+ia}(X^{j+i}) = 0$. Let $j_0\ge 1$ and $b\ge \nu$ be integers. Set
$\rho_0 = r_{j_0}(b)$ and $\rho_1= r_{j_0+1}(b)$. When $j_0\ge 2$ we also write $\rho_{-1} =
r_{j_0-1}(b)$. Assume that
\begin{enumerate}

\smallskip

\item $\alpha_{\rho_0+1}(X^{j_0})\ne 0$,

\smallskip

\item $\alpha_l(X^{j_0})=0$ for all $l\le \rho_0$,

\smallskip

\item $\gamma_s^{(l)}=\delta_{ls}\ev_0$ for all $l> \rho_0$ and $s<l+(\nu-1)b^{j_0-1}$,

\smallskip

\item $\alpha_{1k}=0$ for $2\le k\le (\nu-1)b^{j_0-1}$,

\smallskip

\item $\alpha_{ij}=0$ for $j<b$ and $2\le i\le 1+(\nu-1)b$,

\smallskip

\item If $\nu=2$, then $\alpha_{ijk}=0$ for $2\le i\le 2b-2$, $j\le 2b-2$ and $k<b$,

\smallskip

\item For $j_0\ge 2$, we have $\gamma_{\rho_0+1}^{(\rho_{-1}+1)} = \alpha_{\nu}\circ
\alpha_b^{\rho_{-1}} +\alpha_1\circ \alpha_{(\nu-1)b^{j_0-1}+1}$.

\smallskip

\end{enumerate}
Then,
\begin{enumerate}

\smallskip

\item[(8)] $\alpha_l(X^j)=0$ for all $l\le \rho_1$ and $j>j_0$,

\smallskip

\item[(9)] $\gamma_s^{(l)} = \delta_{ls}\ev_0$ for all $l>\rho_0$ and $s<l+(\nu-1)b^{j_0}$,

\smallskip

\item[(10)] $\alpha_{1k} = 0$ for $2\le k\le (\nu-1)b^{j_0}$,

\smallskip

\item[(11)] $\gamma_{\rho_1+1}^{(\rho_0+1)} = \alpha_{\nu}\circ \alpha_b^{\rho_0}+\alpha_1\circ
\alpha_{(\nu-1)b^{j_0}+1}$,

\smallskip

\item[(12)] For $j_0\ge 2$, we have
$$
\qquad\gamma_{\rho_0+1}^{(\rho_{-1}+1)}=\alpha_{\nu}\circ\alpha_b^{\rho_{-1}}\quad\text{and}
\quad \alpha_{\rho_1+1}(X^{j_0+1})=\alpha_{\rho_0+1}(X^{j_0})\gamma_{\rho_1+1}^{(\rho_0+1)}(X).
$$
\end{enumerate}
\end{lemma}

\begin{proof} By items~(2) and (3) we can apply Lemma~\ref{diagonalestriviales} with
$a=(\nu-1)b^{j_0-1}$ and the same $j_0$ and $\rho_0$. Hence,
$$
\alpha_{l}(X^{j+i})=0\quad\text{for $i\ge 0$, $l\le \rho_0+i(\nu-1)b^{j_0-1}$ and $j\ge j_0$.}
$$
In particular $\alpha_{l}(X^{j+1})=0$ for $l\le \rho_0+(\nu-1)b^{j_0-1}$ and $j\ge j_0$. From
this and items~(3) and~(4), it follows that in order to prove items~(8)--(10), it suffices to
prove by induction on $k=(\nu-1)b^{j_0-1}+1,\dots,(\nu-1)b^{j_0}$ that
\begin{equation}
\begin{aligned}
& \alpha_{1k}=0,\\
& \alpha_{\rho_0+k}(X^{j+1})=0&&\quad\text{for $j\ge j_0$},\\
& \gamma_s^{(l)}=\delta_{ls}\ev_0&&\quad\text{for $l>\rho_0$ and $s<l+k$}
\end{aligned}
\label{eqB}
\end{equation}
By the discussion above and conditions~(3) and~(4), we know that~\eqref{eqB} is true for
$k=(\nu-1)b^{j_0-1}$. Suppose it is true for $k \le k_1$, where $k_1$ is a fix integer
satisfying
$$
(\nu-1)b^{j_0-1}\le k_1<(\nu-1)b^{j_0}.
$$
In order to perform the inductive step it suffices to apply Lemma~\ref{diagonales} with $a=k_1$
and the same $j_0$ and $\rho_0$. Conditions~(1) and~(2) of that lemma are the corresponding
assumptions of the present lemma and condition~(3) follows from the inductive hypothesis. The
fact that condition~(4) is also satisfied (i.e. $\gamma^{(l)}_{l+k_1}=\alpha_1\circ
\alpha_{k_1+1}$ for all $l>\rho_0$) follows from item~(3) of Lemma~\ref{alfasygammas}, which
applies thanks to the inductive hypothesis and assumptions~(5) and (6). Item~(11) also follows
from item~(3) of Lemma~\ref{alfasygammas}. Finally, the first assertion of item~(12) follows
from items~(7) and~(10), since clearly $(\nu-1)b^{j_0-1}+1\le (\nu-1)b^{j_0}$. The second
assertion is a direct consequence of the equality
$$
\alpha_{\rho_1+1}(X^{j_0+1})= \sum_{l=0}^{\rho_1+1} \alpha_l(X^{j_0})\gamma_{\rho_1+1}^{(l)}(X)
$$
and items~(2) and~(9), since $\rho_1+1<l+(\nu-1)b^{j_0}$ for all $l>\rho_0+1$.
\end{proof}

\begin{lemma}\label{nigualados} Suppose $\nu=2$ and $b\ge 4$. Assume that for each integer $a$
satisfying $b\le a < 2b-3$ and $u\ge 1$ there is an $i$ such that $\alpha_{b+ia+2}(X^{iu+2}) =
0$. Write $\alpha_k(X) = \sum_{i\ge 0} q_{ik} X^i$. If
\begin{enumerate}

\smallskip

\item $\alpha_{b+2}(X^2)\ne 0$,

\smallskip

\item $\alpha_k(X^2)=0$ for $k\le b+1$,

\smallskip

\item $\alpha_{1k} = 0$ for $2\le k\le b$,

\smallskip

\item $q_{1k}=0$ for $k<b$,

\smallskip

\end{enumerate}
then $\alpha_{ijk}=0$ for $2\le i\le 2b-2$, $j\le 2b-2$ and $k<b$.
\end{lemma}

\begin{proof} In order to prove the lemma we will need to establish some auxiliary facts. We
first prove that
\begin{equation}
\alpha_{ij} = 0\qquad\text{for $2\le i \le b+1$ and $j<b$.}\label{eqT}
\end{equation}
This is immediate when $j=1$, since $\alpha_1 = \ev_0$ and $\alpha_i(1) = 0$. Thus, we can
assume $j>1$. By item~(2) and Lemma~\ref{nuevo},
\begin{equation}
\alpha_h(X^r) = 0\quad\text{for all $h\le b+1$ and $r\ge 2$}\label{eqP}
\end{equation}
and so $\alpha_{ij}(X^r) = 0$ for $r\ge 2$. Since also $\alpha_{ij}(1) = 0$, we are reduced to
prove that $\alpha_{ij}(X) = 0$. By item~(3) and the fact that $\alpha_1 = \ev_0$,
$$
0 = \alpha_{1j}(X) = \sum_{i\ge 0} q_{ij} \alpha_1(X^i) = q_{0j}.
$$
Hence, by item~(4) we have $\alpha_j(X)\in X^2k[X]$. Thus, applying~\eqref{eqP} with $i$
instead of $h$, we obtain $\alpha_{ij}(X) = 0$. We now prove that
\begin{equation}
\gamma^{(l)}_s = \delta_{ls}\ev_0\qquad\text{for $l\ge b$ and $s<l+b$.}\label{eqW}
\end{equation}
Since
$$
\gamma^{(l+1)}_s = \sum_{j=1}^s \alpha_j\circ \gamma^{(l)}_{s-j}\ ,
$$
it will be sufficient to prove this for $l = b$ and $s<2b$. But using that $\alpha_{1j}=
\alpha_{j1}=0$ for $2\le j\le b$ and $\alpha_{ij}=0$ for $2\le i\le b$ and $j<b$, it is easy to
see that
$$
\gamma^{(b)}_s =\sum_{|i_1,\dots, i_b| = s} \alpha_{i_1\cdots i_b} = \delta_{bs}\ev_0.
$$
Finally, we also need to check that
\begin{equation}
\alpha_{1k}=0\qquad\text{for $k=2,\dots,2b-3$.}\label{eqR}
\end{equation}
By item~(3) we must show that
$$
\alpha_{1k}=0\qquad\text{for $k=b+1,\dots,2b-3$.}
$$
To check this it suffices to prove by induction on $k=b,\dots,2b-3$, that
\begin{equation}
\alpha_{1k}=0\qquad\text{and}\qquad \gamma_s^{(l)}=\delta_{ls}\ev_0\label{eqC}
\end{equation}
for $l>b+1$ and $s<l+k$. By item~(3) and~\eqref{eqW}, we know that~\eqref{eqC} is true for
$k=b$. Suppose it is true for $k \le k_1$, where $b\le k_1<2b-3$. In order to perform the
inductive step it suffices to apply Lemma~\ref{diagonales} with $a=k_1$, $\rho_0=b+1$ and
$j_0=2$. Conditions~(1) and~(2) of that lemma are the corresponding assumptions of the present
lemma and condition~(3) follows from the inductive hypothesis. The fact that condition~(4) is
also satisfied (i.e. $\gamma^{(l)}_{l+k_1}=\alpha_1\circ \alpha_{k_1+1}$ for all $l>b+1$),
follows from item~(2) of Lemma~\ref{alfasygammas}, which applies thanks to the inductive
hypothesis and~\eqref{eqT}.

We now are ready to prove the thesis. By item(2) of Theorem~\ref{teorema 2.1}, item~(2) and
Lemma~\ref{nuevo}, we only need to check that
$$
\alpha_{ijk}(X)=0\quad\text{for $2\le i\le 2b-2$, $j\le 2b-2$ and $k<b$.}
$$
For this will be sufficient to prove that
\begin{xalignat}{2}
&\alpha_i(X^r)=0 &&\quad\text{for $i\le 2b-2$ and $r\ge 3$,}\label{eqG}\\
&\alpha_{jk}(X)\in X^3k[X]&&\quad\text{for $j\le 2b-2$ and $k<b$.}\label{eqH}
\end{xalignat}
From~\eqref{eqW} and item~(2), it follows that the hypothesis of
Lemma~\ref{diagonalestriviales} are satisfied with $\rho_0 = b+1$, $j_0 = 2$ and $a = b$.
Thus~\eqref{eqG} is true. We next prove~\eqref{eqH}. Let $k< b$ and $j\le 2b-2$. Since $q_{0k}
= \alpha_{1k}(X) = 0$ for $k\le b$, from~\eqref{eqG} we get
$$
\alpha_{jk}(X) = \sum_{r\ge 1} q_{rk}\alpha_j(X^r) = q_{1k}\alpha_j(X) + q_{2k}\alpha_j(X^2).
$$
But $q_{1k} = 0$ for $k<b$ by item~(4). Thus $ \alpha_{jk}(X) = q_{2k}\alpha_j(X^2)$, and so we
are reduced to prove that $\alpha_j(X^2)\in X^3k[X]$. Now,
\begin{equation}
\alpha_j(X^2) = \sum_{l=2}^j \alpha_l(X)\gamma_j^{(l)}(X) = \sum_{l=2}^{\min(j,b-1)}
\alpha_l(X) \gamma_j^{(l)}(X),\label{eqI}
\end{equation}
where the first equality follows from the fact that $\alpha_0 = 0$ and $\alpha_1(X) = 0$, and
the second one from the fact that $\gamma_j^{(l)}= \delta_{jl}\ev_0$ for $b\le l\le j\le 2b-2$,
by~\eqref{eqW}. Since $q_{0l}=q_{1l} = 0$, we know that
\begin{equation}
\alpha_l(X)\in X^2k[X]\quad\text{for $l<b$.}\label{eqJ}
\end{equation}
So, if we prove that $\gamma_j^{(l)}(X)\in Xk[X]$ for $j\le 2b-2$ and $2\le l\le \min(j,b-1)$,
the fact that $\alpha_j(X^2)\in X^3k[X]$ follows. Clearly $\gamma_l^{(l)}(X)= 0$. So we can
assume $j>l$. By definition
$$
\gamma_j^{(l)}(X)= \sum_{|n_1,\dots, n_l|=j}\alpha_{n_1\dots n_l}(X).
$$
By~\eqref{eqR} we know that $\alpha_{1k}=0$ for $1<k < 2b-2$. Thus, if $\alpha_{n_1\dots
n_l}(X) \ne 0$, then there cannot be any $n_i=1$.  Write
$$
\alpha_{n_3\dots n_l}(X)= \sum_{i=0}^u p_iX^i\quad\text{and}\quad \alpha_{n_2\dots n_l}(X)=
\sum_{i=0}^v p'_iX^i.
$$
To finish the proof we will use that
\begin{equation}
\alpha_n(X^2k[X]))\subseteq X^2k[X]\qquad\text{for $2\le n \le 2b-2$},\label{eqK}
\end{equation}
which follows immediately from the following facts:
\begin{alignat*}{2}
& \alpha_n(X^r) = 0 \text{ for all $r>2$}&&\qquad\text{by~\eqref{eqG}},\\
& \alpha_n(X^2)\in X^2k[X]&&\qquad\text{by~\eqref{eqI} and~\eqref{eqJ}.}
\end{alignat*}
We consider two cases: If $n_2<b$, then by the fact that $\alpha_{n_2}(1) = 0$ and
equalities~\eqref{eqJ}, \eqref{eqK},
$$
\alpha_{n_1\dots n_l}(X)= \sum_{i=0}^u p_i\alpha_{n_1n_2}(X^i)\in X^2k[X].
$$
If $b\le n_2\le 2b-2$, then $n_1<b-1$, and so by the fact that $\alpha_{n_1}(1) = 0$, and
equalities~\eqref{eqJ}, \eqref{eqK},
$$
\alpha_{n_1\dots n_l}(X)= \sum_{i=0}^v p'_i \alpha_{n_1}(X^i) \in X^2k[X],
$$
as desired.
\end{proof}

\begin{lemma}\label{previoanmayorquedos} Suppose there exists $m$ such that $\alpha_k=0$ for
all $k\ge m$. Recall that $\alpha_k(X) = \sum_i q_{ik} X^i$. Let $u\in \{1,2\}$. Then the
following facts hold:
\begin{enumerate}

\smallskip

\item $\alpha_l(X^u)=0$ for $l<\nu^u$,

\smallskip

\item $\gamma_s^{(l)} = \delta_{ls}\ev_0$ for $l\ge\nu^u$ and $s<l+(\nu-1)\nu^{u-1}$,

\smallskip

\item $\alpha_{1k}=0$ for $2\le k\le (\nu-1)\nu^{u-1}$,

\smallskip

\item $\alpha_{ij}=0$ for $j<\nu$ and $2\le i\le 1+(\nu-1)\nu$,

\smallskip

\item If $\nu=2$ then $\alpha_{ijk}=0$ for $2\le i\le 2\nu-2$, $j\le 2\nu-2$ and $k<\nu$,

\smallskip

\item For $u\ge2$, we have $\gamma_{\nu^u}^{(\nu^{u-1})} = \alpha_{\nu}^{\nu^{u-1}} +\alpha_1\xcirc
\alpha_{(\nu-1)\nu^{u-1}+1}$

\smallskip

\item $\alpha_{\nu^2}(X^2) = q_{1\nu} = q_{0,(\nu-1)\nu+1} = 0.$

\end{enumerate}
\end{lemma}

\begin{proof} By hypothesis there exists $u_0$ such that $\alpha_{\nu^{u_0}}(X^{u_0}) = 0$ and
$\alpha_{\nu^u}(X^u) \ne 0$ for $1\le u < u_0$. By the discussion preceding
Lemma~\ref{alfasygammas} we know that $u_0\ge 2$.  Next we prove by induction on $1\le u< u_0$
that items~(1)--(6) are valid for $1\le u\le u_0$. Then we will see that $u_0=2$ and $q_{1\nu}
= q_{0,(\nu-1)\nu+1} = 0$. When $u = 1$ item~(6) is empty and items~(1), (3) and~(4) are
satisfied, since $\alpha_1(X)=0$, $\alpha_l=0$ for $1<l<\nu$ and $\alpha_{i1} = 0$ for all
$i>1$. Item~(5) is a direct consequence of item~(4). Finally, by item~(1) of
Lemma~\ref{alfasygammas} we have $\gamma_{l+k}^{(l)} = \alpha_{1,k+1}=0$ for $1\le k<\nu-1$ and
$l\ge \nu$, from which item~(2) follows easily. Suppose the result is valid for a fixed
$u<u_0$. We will apply Lemma~\ref{curvas} with $j_0 = u$ and $b = \nu$. Note that in this case
$\rho_0 = \nu^u-1$, $\rho_1 = \nu^{u+1}-1$ and $\rho_{-1} = \nu^{u-1}-1$ (if $u\ge 2$).
Item~(1) of Lemma~\ref{curvas} is valid by the definition of $u_0$ and items~(2)--(7) are
items~(1)--(6) above. To perform the inductive step it suffices to note that items~(4) and~(5)
above do not depend on $u$, and that items~(8)--(11) of Lemma~\ref{curvas} implies items~(1),
(2), (3) and~(6) above with $u+1$ instead of $u$. Hence items~(1)--(6) are valid for $1\le u\le
u_0$. Moreover, item~(12) of Lemma~\ref{curvas} gives
\begin{equation}
\gamma_{\nu^u}^{(\nu^{u-1})} = \alpha_{\nu}^{\nu^{u-1}}\quad\text{and}\quad
\alpha_{\nu^{u+1}}(X^{u+1}) = \alpha_{\nu^u}(X^u) \gamma_{\nu^{u+1}}^{(\nu^u)}(X),\label{eqFF}
\end{equation}
for $2\le u<u_0$. Since $\alpha_{\nu^{u_0}}(X^{u_0}) = 0$ and
$\alpha_{\nu^{u_0-1}}(X^{u_0-1})\ne 0$, the last equality yields
\begin{equation}
\gamma_{\nu^{u_0}}^{(\nu^{u_0-1})}(X) = 0.\label{eqEE}
\end{equation}
In the sequel we are going to use the equality
\begin{equation}
q_{0k} = \sum_i q_{ik} \ev_0(X^i) = \alpha_{1k}(X)\qquad\text{for $k\ge 1$.}\label{eqBB}
\end{equation}
By this equality and item~(3) with $u=2$, we have $q_{0\nu} = 0$. Moreover item~(1) (also with
$u=2$) implies \hbox{$\alpha_l(X^2) = 0$} for all $l\le \nu$. So,
$$
\alpha_{\nu}(X^i) = \sum_{l=1}^{\nu}\alpha_l(X^2) \gamma^{(l)}_{\nu}(X^{i-2}) = 0\quad\text{for
all $i\ge 2$.}
$$
Thus, for $u< u_0$,
\begin{equation}
\alpha_{\nu}^{\nu^u}(X) = \sum_{i>0} q_{i\nu} \alpha_{\nu}^{\nu^u-1}(X^i) = q_{1\nu}
\alpha_{\nu}^{\nu^u-1}(X) = \dots = q_{1\nu}^{\nu^u-1} \alpha_{\nu}(X).\label{eqAA}
\end{equation}
Hence, by item~(6) with $u=u_0$ and equalities~\eqref{eqEE}, \eqref{eqBB} and \eqref{eqAA},
\begin{align*}
q_{1\nu}^{\nu^{u_0-1}-1} \alpha_{\nu}(X) + q_{0,(\nu-1)\nu^{u_0-1}+1} & = \alpha_{\nu}^{
\nu^{u_0-1}}(X) + \alpha_1\xcirc \alpha_{(\nu-1)\nu^{u_0-1}+1}(X)\\
& = \gamma_{\nu^{u_0}}^{\nu^{u_0-1}}(X)\\
& = 0.
\end{align*}
Since $\deg(\alpha_{\nu}(X))>0$ (because $q_{0\nu} =0$ and $\alpha_{\nu}(X)\ne 0$), this
implies
\begin{equation}
q_{0,(\nu-1)\nu^{u_0-1}+1} = q_{1\nu} = 0.\label{eqHH}
\end{equation}
Now, we claim that $u_0 = 2$ (that is $\alpha_{\nu^2}(X^2) = 0$). In fact, by~\eqref{eqFF},
\eqref{eqAA} and~\eqref{eqHH},
$$
\alpha_{\nu^2}(X^2) = \alpha_{\nu}(X) \gamma_{\nu^2}^{(\nu)}(X) = \alpha_{\nu}(X)
\alpha_{\nu}^{\nu}(X) = \alpha_{\nu}(X)^2 q_{1\nu}^{\nu-1} = 0,
$$
as we want.
\end{proof}

\begin{lemma}\label{nmayorquedos} Under the hypothesis of Lemma~\ref{previoanmayorquedos}, we
have:
$$
q_{0k} = q_{1k}=0 \quad\text{for all $k$.}
$$
\end{lemma}

\begin{proof} We are going to prove by induction on $d$ that for all $d\ge\nu$, the following
statements are true
\begin{alignat}{2}
&\alpha_k(X^2)=0 &&\qquad \text{ for $k\le 1+(\nu-1)(d+1)$.}\label{eqA1}\\
&\alpha_{1k}=0 && \qquad\text{ for $2\le k\le (\nu-1)d+1$.}\label{eqA2}\\
& q_{1k} = 0 && \qquad\text{ for $k\le d$.}\label{eqA3}
\end{alignat}
This will be finish the proof of the lemma since $q_{0k} = \alpha_{1k}(X)$ by
equality~\eqref{eqBB}. Assume $d = \nu$. Since $\alpha_1 = \ev_0$ and $\alpha_k = 0$ for $2\le
k <\nu$, we have $q_{1k}= 0$ for $k<\nu$. Combining this with Lemma~\ref{previoanmayorquedos}
we obtain~\eqref{eqA3}. Equality~\eqref{eqA1} follows from items~(1) and~(7) of
Lemma~\ref{previoanmayorquedos}. It remains to check that equality~\eqref{eqA2} holds. For
$2\le k\le (\nu-1)\nu$ this follows from item~(3) of Lemma~\ref{previoanmayorquedos}. To finish
we must see that $\alpha_{1,(\nu-1)\nu+1}(X^r) = 0$ for all $r\ge 0$. For $r = 0$ this follows
from item~(2) of Theorem~\ref{teorema 2.1} and for $r=1$, from item~(7) of
Lemma~\ref{previoanmayorquedos} and the fact that $q_{0,(\nu-1)\nu+1} =
\alpha_{1,(\nu-1)\nu+1}(X)$. Finally, for $r\ge 2$,
$$
\alpha_{1,(\nu-1)\nu+1}(X^r) = \sum_{l = 1}^{(\nu-1)\nu+1}\alpha_{1l}(X)
\alpha_1(\gamma^{(l)}_{(\nu-1)\nu+1}(X^{r-1})) = 0,
$$
where the last equality follows from the fact that $\alpha_{1l}(X) = 0$ for $1\le l\le
(\nu-1)\nu+1$. Suppose that equalities~\eqref{eqA1}, \eqref{eqA2} and~\eqref{eqA3} are true
for~$d$. We claim that if $\nu = 2$, then
$$
\gamma^{(l)}_{l+h} = \delta_{0h}ev_0\qquad\text{for $l\ge 2$ and $h\le d$.}
$$
For $h=0$ this is immediate. Assume $h>0$. Since $\alpha_{k1} = 0$ for all $k\ge 2$ and
$\alpha_{1k} = 0$ for $2\le k\le d+1$, we have
$$
\gamma^{(l)}_{l+h} = \sum_{|n_1,\dots, n_l| = l+h\atop n_1,\dots,n_l\ge 2} \alpha_{n_1\dots
n_l}.
$$
Because of $h\le d$, each one of the indices $(n_1,\dots,n_l)$ in the above sum satisfies $2\le
n_j\le d < r_2(d)$ for all~$j$. Since, by Theorem~\ref{teorema 2.1}, Lemma~\ref{nuevo} and
equalities~\eqref{eqA1},~\eqref{eqA2} and~\eqref{eqA3},
\begin{alignat}{2}
& \alpha_h(1) = 0&& \qquad\text{for $h\ge 2$,}\label{eqA4}\\
& \alpha_h(X^j) = 0&& \qquad\text{for $h\le d+1$ and $j\ge 2$,}\label{eqA5}\\
&\alpha_{hl}(X) = 0 && \qquad\text{for $h\le d+1$ and $l\le d$,}\label{eqA6}
\end{alignat}
the claim is true. Hence, when $\nu = 2$,
\begin{align*}
\alpha_{d+3}(X^2) &= \alpha_2(X)\gamma^{(2)}_{d+3}(X)\\
& = \alpha_2(X)\bigl(\alpha_{1,d+2}(X) +\alpha_{2,d+1}(X) \bigr)\\
& = \alpha_2(X)\bigl(\alpha_{1,d+2}(X)+q_{1,d+1} \alpha_2(X)\bigr),
\end{align*}
where the second equality follows from~\eqref{eqA6} and the fact that $\alpha_1 = \ev_0$, and
the last one from~\eqref{eqA4} and~\eqref{eqA5}. Suppose $\alpha_{d+3}(X^2) = 0$ (that is,
equality~\eqref{eqA1} is valid for $d+1$). Then $\alpha_{1,d+2}(X) = q_{1,d+1} = 0$, since
$\deg(\alpha_2(X))>0$. In particular condition~\eqref{eqA3} is satisfied for $d+1$. Moreover,
$$
\alpha_{1,d+2}(X^i) = \sum_{l=1}^{d+2}\alpha_{1l}(X)\alpha_1\xcirc \gamma^{(l)}_{d+2}(X^{i-1})
= \sum_{l=1}^{d+1}\alpha_{1l}(X)\alpha_1\xcirc \gamma^{(l)}_{d+2}(X^{i-1}) = 0,
$$
for all $i\ge 1$, where the last equality follows from the fact that $\alpha_{11}(X)  = 0$ and
equality~\eqref{eqA2} for $d$. Since, by item~(2) of Theorem~\ref{teorema 2.1}, we also have
$\alpha_{1,d+2}(1) = 0$, condition~\eqref{eqA2} is also valid for $d+1$. Hence in the case that
$\nu=2$ and $\alpha_{d+3}(X^2)= 0$, we are done. For the rest of the proof we assume that
$\nu\ge 2$ and $\alpha_{d+3}(X^2)\ne 0$ if $\nu=2$. For all $h$, let $r_h = r_h(d+1)$. Note
that $r_1 = \nu-1$ and $r_2 = (\nu-1)(d+2)$. By hypothesis there exists $h_0$ such that
$\alpha_{r_{h_0}+1}(X^{h_0}) = 0$ and $\alpha_{r_h+1}(X^h) \ne 0$ for $1\le h <h_0$. By the
discussion preceding Lemma~\ref{alfasygammas} we know that $h_0\ge 2$. Next we prove by induction
on $1\le h< h_0$ that
\begin{enumerate}

\smallskip

\item $\alpha_l(X^h)=0$ for all $l\le r_h$,

\smallskip

\item $\gamma_s^{(l)} = \delta_{ls}\ev_0$ for all $l>r_h$ and $s<l+(\nu-1)(d+1)^{h-1}$,

\smallskip

\item $\alpha_{1k}=0$ for $2\le k\le (\nu-1)(d+1)^{h-1}$,

\smallskip

\item $\alpha_{ij}=0$ for $j\le d$ and $2\le i\le 1+(\nu-1)(d+1)$,

\smallskip

\item If $\nu=2$, then $\alpha_{ijk}=0$ for $2\le i\le 2d$, $j\le 2d$ and $k\le d$,

\smallskip

\item For $h\ge 2$, we have $\gamma_{r_h+1}^{(r_{h-1}+1)} = \alpha_{\nu}\circ
\alpha_{d+1}^{r_{h-1}} +\alpha_1\circ \alpha_{(\nu-1)(d+1)^{h-1}+1}$.

\smallskip

\end{enumerate}
for $1\le h\le h_0$. When $h=1$, items~(1), (2) and (3) do not depend on $d$ and they follow
from items~(1), (2) and (3) of Lemma~\ref{previoanmayorquedos}, respectively. Since item~(6)
does not apply, we only need to verify items~(4) and~(5). Note that by item~(2) of
Theorem~\ref{teorema 2.1}, Lemma~\ref{nuevo} and equality~\eqref{eqA1}, we have
\begin{equation}
\alpha_k(X^s) = 0 \qquad\text{for $2\le k\le 1+(\nu-1)(d+1)$ and $s\ne 1$.}\label{eqII}
\end{equation}
This immediately gives
$$
\alpha_{ij}(X^s) = 0\qquad\text{for $2\le j\le d$ and $s\ne 1$.}
$$
But, by~\eqref{eqA3} and~\eqref{eqII}
$$
\alpha_{ij}(X) = \sum_{s\ge 0} q_{sj}\alpha_i(X^s) = q_{1j}\alpha_i(X) = 0,
$$
for $j\le d$ and $2\le i\le 1+(\nu-1)(d+1)$. So,
$$
\alpha_{ij} = 0\qquad\text{for $j\le d$ and $2\le i\le 1+(\nu-1)(d+1)$,}
$$
since clearly $\alpha_{i1} = 0$ for $i\ge 2$. This establishes item~(4). Suppose $\nu = 2$. If
$d<3$, then item~(5) follows from item~(4), and if $d\ge 3$, from Lemma~\ref{nigualados} with
$b= d+1$, which applies because condition~(1) is satisfied since $\alpha_{d+3}(X^2)\ne 0$ by
assumption and conditions~(2), (3) and~(4) are equalities~\eqref{eqA1},~\eqref{eqA2}
and~\eqref{eqA3}. Assume that items~(1)--(6) are valid for a fix $h<h_0$. We will apply
Lemma~\ref{curvas} with $j_0 = h$ and $b = d+1$. Note that in this case $\rho_0 = r_h$, $\rho_1
= r_{h+1}$ and $\rho_{-1} = r_{h-1}$ (if $h\ge 2$). Item~(1) of Lemma~\ref{curvas} is valid by
the definition of $h_0$ and items~(2)--(7) are items~(1)--(6) above. To perform the inductive
step it suffices to note that items~(4) and~(5) above do not depend on $h$, and that
items~(8)--(11) of Lemma~\ref{curvas} imply items~(1), (2), (3) and~(6) above for $h+1$ instead
of $h$. So, we have established items~(1)--(6) above for $h+1$. Indeed, item~(9) of
Lemma~\ref{curvas} gives the following equality
\begin{equation}
\gamma_s^{(l)} = \delta_{ls}\ev_0\qquad\text{for all $l>r_h$ and $s<l+(\nu-1)(d+1)^h$,}
\label{eqpppp}
\end{equation}
which is stronger that item~(2) above for $h$ and also for $h+1$. Note that~\eqref{eqpppp} for
$h=1$ implies $\gamma_{r_2+1}^{(l)}(X)=0$ for $l>r_1+1 = \nu$. Hence,
\begin{equation}
\alpha_{r_2+1}(X^2) = \sum_{l=1}^{r_2+1}\alpha_l(X) \gamma_{r_2+1}^{(l)}(X) =
\alpha_{\nu}(X)\gamma_{r_2+1}^{(\nu)}(X),\label{eqJJ}
\end{equation}
since $\alpha_l(X)=0$ for $l<\nu$. Moreover, by equality~\eqref{eqII}, we have
$$
\alpha_{\nu}(X^s) = \alpha_{d+1}(X^s) = 0\qquad\text{for $s\ne 1$,}
$$
and so, arguing as in the proof of equality~\eqref{eqAA}, we obtain
\begin{equation}
\alpha_{\nu}\circ\alpha_{d+1}^l(X) = \alpha_{\nu}(X) q_{1,d+1}^l\quad\text{for all $l\ge
1$}.\label{eqKK}
\end{equation}
By item~(6) above for $h = h_0$ and equality~\eqref{eqKK}, we have
\begin{equation}
\gamma_{r_{h_0}+1}^{(r_{h_0-1}+1)}(X) = \alpha_{\nu}(X) q_{1,d+1}^{r_{h_0-1}} + \alpha_1\circ
\alpha_{(\nu-1)(d+1)^{h_0-1}+1}(X).\label{eqLL}
\end{equation}
We claim that $\alpha_{r_2+1}(X^2) = 0$ (Combining this with item~(1) above for $h=2$, we
obtain~\eqref{eqA1} for $d+1$). Assume on the contrary that $\alpha_{r_2+1}(X^2) \ne 0$ (which
implies $h_0\ge 3$). By~\eqref{eqKK} and item~(12) of Lemma~\ref{curvas} with $j_0 = 2$ and $b
= d+1$ (item~(1) of this lemma is satisfied since $\alpha_{r_2+1}(X^2) \ne 0$, and
items~(2)--(7) are items~(1)--(6) above for $h=2$), we have
\begin{equation}
\gamma_{r_2+1}^{(r_1+1)}(X) = \alpha_{\nu}\xcirc \alpha_{d+1}^{r_1}(X) = \alpha_{\nu}(X)
q_{1,{d+1}}^{r_1}.\label{eqMM}
\end{equation}
Combining this with~\eqref{eqJJ}, we obtain $\alpha_{r_2+1}(X^2) = \alpha_{\nu}(X)^2
q_{1,{d+1}}^{r_1}$, which implies $q_{1,d+1}\ne 0$. On the other hand, by item~(12) of
Lemma~\ref{curvas} with $j_0 = h_0-1$ and $b=d+1$, we have
$$
\alpha_{r_{h_0}+1}(X^{h_0}) = \alpha_{r_{h_0-1}+1}(X^{h_0-1})
\gamma_{r_{h_0}+1}^{(r_{h_0-1}+1)}(X).
$$
Since $\alpha_{r_{h_0}+1}(X^{h_0}) = 0$ and $\alpha_{r_{h_0-1}+1}(X^{h_0-1})\ne 0$, this and
equality~\eqref{eqLL} imply
$$
0 = \alpha_{\nu}(X) q_{1,{d+1}}^{r_{h_0-1}}+\alpha_{1,(\nu-1)(d+1)^{h_0-1}+1}(X).
$$
Since $deg(\alpha_{\nu}(X))>0$ and $\alpha_{1,(\nu-1)(d+1)^{h_0-1}+1}(X)\in k$, this yields
$q_{1,d+1}=0$, which contradicts the fact that $q_{1,d+1}\ne 0$. This proves the claim. Thus
$\alpha_{r_2+1}(X^2) = 0$ or, equivalently, $h_0=2$. This, together with equality~\eqref{eqJJ},
gives $\gamma_{r_2+1}^{(r_1+1)}(X) = 0$. Since $deg(\alpha_{\nu}(X))>0$, from
equality~\eqref{eqLL} it follows that
$$
q_{1,d+1} = 0\quad\text{and}\quad \alpha_{1,(\nu-1)(d+1)+1}(X) = 0.
$$
Combining the first equality with~\eqref{eqA3} for $d$, we get~\eqref{eqA3} for $d+1$. Finally,
by item~(2) of Theorem~\ref{teorema 2.1}, we know that $\alpha_{1,(\nu-1)(d+1)+1}(1) = 0$, and
for $r\ge 2$,
$$
\alpha_{1,(\nu-1)(d+1)+1}(X^r) = \sum_{l = 1}^{(\nu-1)(d+1)+1}\alpha_{1l}(X)
\alpha_1(\gamma^{(l)}_{(\nu-1)(d+1)+1}(X^{r-1})) = 0,
$$
where the last equality follows from item~(3) above for $h=2$ and the fact that
$$
\alpha_{11}(X) = \alpha_{1,(\nu-1)(d+1)+1}(X) = 0.
$$
All this, together with item~(3) for $h=2$, yields  equality~\eqref{eqA2} for $d+1$.
\end{proof}

\noindent{\bf Proof of Theorem~\ref{teorema 3.4}.}\enspace Item~(2) follows from item~(1) using
that $s$ is a twisting map if and only if $\tau\xcirc s\xcirc \tau$ is, where $\tau$ denotes
the flip. we now prove item~(1). Since $q_{i0} = q_{i1} = 0$ for all $i$, we have $\alpha_0(X)
= \alpha_1(X) = 0$. An inductive argument using formula~\eqref{eq0} and the fact that
$\alpha_0(1) = 0$ show that $\alpha_0 = 0$. Then, by Remark~\ref{remark 2.2}, we know that
$\alpha_1 = \ev_0$. So, we can apply Lemma~\ref{nmayorquedos} in order to obtain that $q_{0j} =
q_{1j} = 0$ for all $j\ge 0$, as we want.\hfill\qed

\end{document}